\documentclass[reqno, 11pt]{amsart}
\usepackage[top=3.75cm, bottom=3.75cm, left=2.5cm, right=2.5cm, headsep=0.75cm]{geometry}            
\usepackage{amsmath,amsrefs,amssymb}
\usepackage{pgfplots}
 \usepackage{tikz,float}
\pgfplotsset{compat = newest}

\usepackage{graphicx}
\usepackage[colorlinks=true, pdfborder={0 0 0}]{hyperref}
\usepackage{subcaption}
\usepackage[font=scriptsize]{caption}
\usepackage{mwe}

\newcommand{\R}{\mathbb{R}}

\newcommand{\B}{\mathbb{B}}
\newcommand{\Be}{\mathsf{B}}

\newcommand{\eps}{\varepsilon}

\newtheorem{theorem}{Theorem}[section]
\theoremstyle{plain}

\newtheorem{corollary}{Corollary}[section]

\newtheorem{definition}{Definition}[section]

\newtheorem{lemma}{Lemma}[section]

\newtheorem{proposition}{Proposition}[section]
\newtheorem{remark}{Remark}[section]

\usepackage{thmtools}

\newtheorem{theorema}{Theorem}[section]

\newcommand{\hn}{\mathbb{H}^{n}}
\newcommand{\De} {\Delta}
\newcommand{\la} {\lambda}
\newcommand{\rn}{\mathbb{R}^{n}}
\newcommand{\bn}{\mathbb{B}^{n}}

\newcommand{\calu}{\mathcal{U}}
\newcommand{\dvg}{{\rm d}v_{\bn}}
\newcommand{\tstar}{2^{\star}}
\newcommand{\cald}{\mathcal{D}}

\numberwithin{equation}{section}

\allowdisplaybreaks

\begin{document}
\title[ Stability of Poincar\'e-Sobolev inequality in the hyperbolic space]{Sharp Quantitative stability of Poincar\'e-Sobolev inequality in the hyperbolic space and applications to fast diffusion flows}

\author{Mousomi Bhakta}
\address{Mousomi Bhakta, Department of Mathematics\\
Indian Institute of Science Education and Research Pune (IISER-Pune)\\
Dr Homi Bhabha Road, Pune-411008, India}
\email{mousomi@iiserpune.ac.in}
\author{Debdip Ganguly}
\address{Debdip Ganguly, Department of Mathematics\\
Indian Institute of Technology Delhi\\
IIT Campus, Hauz Khas, New Delhi, Delhi 110016, India}
\email{debdipmath@gmail.com, debdip@maths.iitd.ac.in}
\author{Debabrata Karmakar}
\address{Debabrata Karmakar, Tata Institute of Fundamental Research\\
Centre For Applicable Mathematics\\
Post Bag No 6503, GKVK Post Office, Sharada Nagar, Chikkabommsandra, Bangalore 560065, India}
\email{debabrata@tifrbng.res.in}
\author{Saikat Mazumdar}
\address{Saikat Mazumdar, Department of Mathematics\\
Indian Institute of Technology Bombay\\
Mumbai 400076, India}
\email{saikat.mazumdar@iitb.ac.in}

\date{\today}


\begin{abstract}
Consider the Poincar\'e-Sobolev inequality on the  hyperbolic space: for every $n \geq 3$  and $1 < p \leq \frac{n+2}{n-2},$ there exists a best constant 
$S_{n,p, \lambda}(\mathbb{B}^{n})>0$ such that 
$$S_{n, p, \lambda}(\mathbb{B}^{n})\left(~\int \limits_{\mathbb{B}^{n}}|u|^{{p+1}} \, {\rm d}v_{\mathbb{B}^n} \right)^{\frac{2}{p+1}}
 \leq\int \limits_{\mathbb{B}^{n}}\left(|\nabla_{\mathbb{B}^{n}}u|^{2}-\lambda u^{2}\right) \, {\rm d}v_{\mathbb{B}^n},$$
holds for all  $u\in C_c^{\infty}(\mathbb{B}^n),$ and  $\lambda \leq \frac{(n-1)^2}{4},$ where $\frac{(n-1)^2}{4}$ is the bottom of the  $L^2$-spectrum of $-\Delta_{\mathbb{B}^n}.$ It is known from the results of Mancini and Sandeep \cite{MS} that under appropriate assumptions on $n,p$ and $\lambda$ there exists an optimizer, unique up to the hyperbolic isometries, attaining the best constant $S_{n,p,\lambda}(\mathbb{B}^n).$ In this article we investigate the quantitative gradient stability of the above inequality and the associated Euler-Lagrange equation locally around a bubble.
In our first result, we prove a 
\it sharp quantitative stability \rm of the above Poincar\'e-Sobolev inequality: if $u \in H^1(\mathbb{B}^n)$ almost optimizes the above inequality then $u$ is close to the manifold of optimizers in a quantitative way. Secondly, we prove  
the quantitative stability of its critical points: if $u \in H^1({\mathbb{B}^n})$ almost solves the Euler-Lagrange equation corresponding to the above Poincar\'e-Sobolev inequality and the energy of $u$ is close to the energy of an extremizer, then we derive the following quantitative bound 
 $$
 \hbox{dist}\left(u, \mathcal{Z}\right) \leq C(n,p,\lambda) \|\Delta_{\mathbb{B}^n} u + \lambda u + u^p \|_{H^{-1}(\mathbb{B}^n)},
 $$
 where $\mathcal{Z}$ denotes the manifold of non-negative finite energy solutions of
$ -\Delta_{\mathbb{B}^{n}}w-\lambda w=|w|^{p-1}w$. Our result generalizes the sharp quantitative stability of
Sobolev inequality in $\rn$ of Bianchi-Egnell \cite{BE} and Ciraolo-Figalli-Maggi \cite{CFM} to the Poincar\'{e}-Sobolev 
inequality on the hyperbolic space.

 Furthermore, combining our stability results and implementing {\it a refined smoothing estimates}, we prove a \it quantitative extinction rate \rm towards its basin of attraction of the solutions of the sub-critical fast diffusion flow for radial initial data. In another application, we derive sharp quantitative stability of the Hardy-Sobolev-Maz'ya inequalities for the class of \it functions which are symmetric in the component of singularity\rm.

\end{abstract}
\maketitle
\setcounter{tocdepth}{1}
\tableofcontents

\section{Introduction}

The sharp quantitative stability of various functional and geometric inequalities and its applications in the calculus of variations, differential geometry and diffusion flow and many more have generated a lot of interest in recent years. Some classical examples are the Euclidean isoperimetric inequality, Sobolev inequality, Gagliardo-Nirenberg-Sobolev inequality, and Caferalli-Kohn-Nirenberg inequality to name a few from the non-exhaustive list of references \cite{ISO1, ISO2, ISO3, ISO4, BE, BL, CI, CFMP, CFM, CF, DT, DO, FG, FMP, FMP-1, FN, BDNS, Ngu, Ruff, FZp, CKNWei, Weietal}. 
 In this article, we study the quantitative stability of the Poincar\'e-Sobolev inequality on the hyperbolic space, first in the spirit of Bianchi-Egnell's result \cite{BE} and then in the spirit of the results of Ciraolo, Figalli and Maggi \cite{CFM}. The classical Sobolev inequality in $\mathbb{R}^n,$ $n \geq 3,$ asserts that there exists a best constant $S(\R^n)$ such that
\begin{align}\label{Sobolev}
S(\R^{n}) \left(~\int \limits_{\R^{n}}|u|^{2^{\star}}{\rm d}x\right)^{2/2^{\star}} \leq \int \limits_{\R^{n}}|\nabla u|^{2}{\rm d}x
\end{align}
holds for all $u \in C_c^{\infty}(\R^n),$ where $2^{\star} = \frac{2n}{n-2}.$ By density argument, the inequality \eqref{Sobolev} continues to hold for all $u$ 
satisfying $\|\nabla u\|_{L^2(\R^n)} <\infty,$ and $\mathcal{L}^n(\{|u| >t\}) < \infty$ for every $t>0,$ where $\|\cdot \|_{L^2(\R^n)}$ denotes the $L^2$-norm and $\mathcal{L}^n$ denotes the Lebesgue measure on $\R^n.$ 
The value of $S(\R^n)$ is known and the equality cases in \eqref{Sobolev} have been well studied. The inequality \eqref{Sobolev} is invariant under the action of the conformal group of $\rn$ composed of the translations and dilations: for $z \in \rn, \mu > 0,$  and $u \in C_{c}^{\infty}(\rn)$ define $T_{z,\mu}(u) = \mu^{\frac{n-2}{2}}u(\mu(\cdot - z)),$ then both the norms in \eqref{Sobolev} are preserved. It is well known that \eqref{Sobolev}  is achieved if and only if $u$ is a constant multiple of the {\it Aubin-Talenti} bubbles \cite{AT, TAL}
\begin{align}\label{AT bubbles}
U[z,\mu](x) = (n(n-2))^{\frac{n-2}{4}}\mu^{\frac{n-2}{2}} \frac{1}{(1 + \mu^2|x-z|^2)^{\frac{n-2}{2}}}, \quad z \in \R^n, \mu>0. 
\end{align}
Therefore the set of equality cases in \eqref{Sobolev} forms a $(n+2)$ dimensional manifold. The choice of the dimensional constant in the definition of $U[z,\mu] = T_{z,\mu}(U[0,1])$ ensures that $U$ is a positive solution of the corresponding Euler-Lagrange equation 
\begin{align*}
-\Delta U  = U^{2^{\star} - 1} \ \mbox{in} \ \R^n.
\end{align*}
In addition, the inequality \eqref{Sobolev} fails to be true if the exponent $\tstar$ is replaced by any other exponent $p \neq \tstar.$ On the other hand, for a smooth bounded domain $\Omega$  of $\rn$ the $L^2$-spectral gap of the Dirichlet Laplacian and interpolation inequalities imply, for every $1 < p  \leq \tstar-1,$ there exists a best constant $S_p(\Omega)$ such that 
\begin{align*}
S_p(\Omega) \left(~\int \limits_{\Omega}|u|^{p+1}~{\rm d}x\right)^{\frac{2}{p+1}} \leq \int \limits_{\Omega}|\nabla u|^{2}~{\rm d}x
\end{align*}
holds for all $u \in C_c^{\infty}(\Omega).$ Hence by density it continues to hold to the closure $H_0^1(\Omega)$
with respect to the norm $\| \nabla u \|_{L^2(\Omega)}.$ It turns out that the constant $S_{\tstar-1}$ does not depend on the domain $\Omega$ and by conformal invariance of the norms $S_{\tstar-1}(\Omega) = S(\rn)$ holds for every sufficiently regular domain $\Omega.$ 


\medskip

Let us now look at the counterpart of the Sobolev inequality in the hyperbolic space.
Throughout this article, we shall consider the Poincar\'e ball model of the hyperbolic space $\B^{n}$, $n \geq 3$ (see Section \ref{pre} for definition). The Euclidean unit ball  $\Be^{n}:= \{ x \in \R^{n}: |x|<1\}$  endowed with the metric $g_{\B^{n}}= \left(\frac{2}{1-|x|^{2}}\right)^{2}{g_{\text{Eucl}}}$ represents the Ball model for the hyperbolic $n$-space, where $g_{\text{Eucl}}$ is the Euclidean metric.

\medskip

We denote by $\dvg$ the volume element and $\nabla_{\B^{n}}$
the gradient vector field. Then the Sobolev inequality or rather Poincar\'e-Sobolev inequality  obtained by Mancini-Sandeep \cite{MS} takes the following form:


\medskip 

\noindent
{\bf Poincar\'e-Sobolev inequality on the hyperbolic space.} Let $n \geq 3$ and $\lambda \leq \frac{(n-1)^2}{4}.$ Then for every $1 < p \leq \tstar-1$ there exists a best constant $S_{n,p, \lambda}(\B^{n})>0$ such that 
\begin{align}\label{S-inq}
S_{n, p, \lambda}(\B^{n})\left(~\int \limits_{\B^{n}}|u|^{p+1} \, \dvg \right)^{\frac{2}{p+1}}\leq\int \limits_{\B^{n}}\left(|\nabla_{\B^{n}}u|^{2}-\lambda u^{2}\right)\, \dvg
\end{align}
holds for all  $u\in C_c^{\infty}(\B^n).$ 

\medskip

Before we proceed further some comments about \eqref{S-inq} are in order:

\begin{itemize}
\item[(i)] Here $|\nabla_{\B^{n}}u|^{2}$ to be understood with respect to the inner product on the tangent space induced by the metric $g_{\B^{n}}$ (we refer to Section \ref{pre} for appropriate definition).

\medskip

\item[(ii)] By density \eqref{S-inq} continues  to hold for every $u \in H^1(\bn)$ defined by the closure of $C_c^{\infty}(\B^n)$ with respect to the norm $\|\nabla_{\bn}u\|_{L^2(\bn)} + \|u\|_{L^2(\bn)},$ where
$\|\cdot\|_{L^2(\bn)}$ denote the $L^2$-norm with respect to the volume measure.

\medskip 

\item[(iii)] The quantity $ \frac{(n-1)^{2}}{4} =:\lambda_1(\bn)$ is the $L^2$-bottom of the spectrum of $-\Delta_{\B^{n}}$ defined by
$$
\frac{(n-1)^{2}}{4} \, := \, \inf \limits_{u \in H^{1}(\B^{n})\setminus\{0\}}\frac{\displaystyle \int \limits_{\B^{n}}|\nabla_{\B^{n}}u|^{2} \dvg }{\displaystyle \int \limits_{\B^{n}}|u|^{2}\, \dvg},
$$
where $\Delta_{\B^{n}}$ is the Laplace-Beltrami operator. As a result, the restriction on $\lambda$ in \eqref{S-inq} can not be dispensed with.

\medskip

\item[(iv)]  As a consequence of (iii), $\|u\|_{\lambda}$ defined by 
\begin{align} \label{lambda norm}
\|u\|_{\lambda}^2 : = \int \limits_{\B^{n}}\left(|\nabla_{\B^{n}}u|^{2}-\lambda u^{2}\right) \, \dvg~\hbox{ for all } u\in H^1(\B^{n}),
\end{align}
is an equivalent $H^1(\mathbb{B}^n)$ norm for $\lambda < \frac{(n-1)^2}{4}.$ For $\lambda =\frac{(n-1)^2}{4}, \|u \|_{\lambda}$ is no longer an equivalent $H^1$-norm, and in principle one may have $\|u \|_{\lambda} < \infty$ with out being square integrable. Nonetheless, thanks to the above Poincar\'e-Sobolev inequality $\|u \|_{\lambda}$ defines a norm on $C_c^{\infty}(\bn)$ and by density \eqref{S-inq}  can be extended to the closure which is $H^1(\mathbb{B}^n)$ for $\lambda < \frac{(n-1)^2}{4}$ and a much larger space than $H^1(\mathbb{B}^n)$ if $\lambda = \frac{(n-1)^2}{4}.$

\end{itemize}

\medskip

Like the Euclidean case, the Poincar\'e-Sobolev inequality is invariant under the action of the conformal group of the ball model $\bn.$
In the Euclidean space, the translations and the dilations constitute the conformal group of $\rn$ and \eqref{Sobolev} is invariant under their actions. In the case of the hyperbolic space, the conformal group coincides with the isometry group and its elements are given by the \it hyperbolic translations \rm (see Section \ref{pre}).  As a result, both the critical and the subcritical Poincar\'e-Sobolev inequality are invariant under the action of the same conformal group (reminiscent of the translations in the Euclidean setting).

\medskip

In \cite{MS} Mancini and Sandeep also classified the equality cases in \eqref{S-inq}. We will only mention a weaker version of their result below. Suppose an optimizer $\boldsymbol{u}$ exists for \eqref{S-inq}, then taking advantage of the inhomogeneity, if we assume that $\|\boldsymbol{u}\|_{\lambda}^2 = S_{n, p,\lambda}(\B^n)^{\frac{p+1}{p-1}},$ then $\boldsymbol{u}$ solves 
\begin{align}\label{eq1}
 -\Delta_{\B^{n}}u-\lambda u=|u|^{p-1}u~\hbox{ for } u \in H^{1}(\B^{n}).
 \end{align}
  The authors studied the existence and uniqueness of the positive solutions of \eqref{eq1}. According to their results, in the subcritical case (that is $p < \tstar -1$) there always exists a solution but in the critical case (that is $p = \tstar -1$) the existence of positive solutions depend on the dimension as well as on the values of $\lambda.$ To state their result and to simplify the notational complexity of this article, we fix once and for all the assumptions on $\lambda$ and $p.$ We assume
\begin{align*}
\mbox{{(\bf H1)}} \ \ \ \ \ \ \ \ \ \ \ \ \ \ \ \ \ \ \
\begin{cases}
\lambda < \frac{(n-1)^2}{4}, \ \ \ \mbox{when} \ 1 < p < \frac{n+2}{n-2}, \ \mbox{and} \ n \geq 3, \\ \\
\frac{n(n-2)}{4} < \lambda < \frac{(n-1)^2}{4}, \ \ \ \mbox{when} \  p = \frac{n+2}{n-2}, \ \mbox{and} \ n \geq 4. \ \ \ \ \ \ \ \
\end{cases}
\end{align*}


\begin{theorema}[Mancini-Sandeep \cite{MS}]\label{MSthm}
Let $n \geq 3$ and  $\lambda, p$ satisfies the hypothesis {\bf (H1)}. Then there exists a smooth,  positive, radially symmetric and decreasing $\mathcal{U}\in H^{1}(\B^{n})$  which attains the Poincar\'e-Sobolev inequality \eqref{S-inq}. 

Moreover, all the extremals of \eqref{S-inq} are unique up to hyperbolic translations $\tau_{\xi}$, $\xi \in \B^{n}$ and multiplications by constants. 
\end{theorema}

In Theorem \ref{MSthm}, the restrictions on $\lambda$ and the dimension $n$ are sharp, in the sense that if one or more assumptions are violated then the corresponding Euler-Lagrange equation does not admit any positive solution in $H^1(\B^n)$ (see \cite{MS} and Section \ref{key lem} below for more details). At the borderline case $\lambda = \frac{(n-1)^2}{4},$ a positive solution of \eqref{eq1} exists but fails to be square-integrable and hence not an element of $H^1(\bn).$ Moreover it has been shown in \cite{DG1} that the assumption on $\lambda$ in the critical case is sharp even for any solutions (not necessarily positive solution).  For a given $n,p$ and $\lambda$ as in Theorem \ref{MSthm}, we denote the unique extremal in \eqref{S-inq} by $\calu$ (which of course depends on $n,p,\lambda$) and satisfy the following properties
\begin{align} \label{calu}
\calu \in H^1(\bn), \calu > 0, \ \mbox{radially symmetric and decreasing},\, \|\calu\|_{\lambda}^2 = S_{n, p,\lambda}(\B^n)^{\frac{p+1}{p-1}}.
\end{align}

\medskip

In this article, we study the quantitative stability of the above Poincar\'e-Sobolev inequality \eqref{S-inq} by controlling the deviation of a given function from attaining the equality in \eqref{S-inq}, by its distance to the family of extremal functions, in the spirit of the results by Bianchi and Egnell \cite{BE}.  Many new challenges need to be addressed to obtain Bianchi-Egnell type results on the hyperbolic space.  Below we shall describe it.  According to Mancini-Sandeep \cite{MS}, the extremals for \eqref{S-inq} are unique up to hyperbolic translations $\tau_{\xi}$, $\xi \in \B^{n}$ and scalar multiplications, which allows us to define the $(n+1)$-dimensional manifold 
\begin{align*}
\mathcal{Z}_{0}:=\{c~\mathcal{U}\circ\tau_{\xi}:c\in \R\setminus\{0\}, \xi \in \B^{n}\}
\end{align*}
consisting of all the extremals for \eqref{S-inq}. 
\medskip

Our first result is the stability of Poincar\'e-Sobolev inequality in the spirit of Bianchi-Egnell \cite{BE}. Generally speaking, if $u \in H^1(\B^n)$ almost optimizes \eqref{S-inq} then $u$ is almost a constant multiple of $\mathcal{U} \circ \tau_{\xi}$ in a quantitative way. For that matter as in \cite{BE} we define the deficit in Poincar\'e-Sobolev inequality  by
\begin{align*}
\delta(u) &= \left(\|u\|_{\lambda}^2 - S_{n, p,\lambda}(\B^n)\|u\|_{L^{p+1}(\B^n)}^2\right)^{\frac{1}{2}} \notag \\
&=\left[~\int \limits_{\B^{n}}\left(|\nabla_{\B^{n}}u|^{2}-\lambda u^{2}\right) \dvg \, - \, S_{n, p,\lambda}(\B^{n})\left(~\int \limits_{\B^{n}}|u|^{p+1}~\dvg \right)^{\frac{2}{p+1}}\right]^{\frac{1}{2}},
\end{align*}

and the distance of $u$ from the manifold $\mathcal{Z}_0$ by
\begin{align} \label{distance}
dist\left(u, \mathcal{Z}_{0}\right):=\inf \limits_{c\in \R,\, \xi \in \B^{n}}\|u-c~\mathcal{U}\circ\tau_{\xi}\|_\lambda,
\end{align}
where $\| u \|_{\lambda}$ is defined by \eqref{lambda norm}. Note that the distance \eqref{distance} satisfies the scaling property:
$dist\left(\kappa u, \mathcal{Z}_{0}\right) = |\kappa|dist\left(u, \mathcal{Z}_{0}\right)$ for any $\kappa \neq 0.$

\medskip

Our first result asserts that the deficit functional controls the distance of $u$ from the manifold consisting of extremal functions. 

\medskip

\begin{theorem}\label{maintheorem-1}
Let $n \geq 3$ and $\lambda, p$ satisfies the hypothesis {\bf (H1)}. Then there exists a constant $C(n, p,\lambda)>0$ such that for all $u\in H^1(\B^{n})$ we have 
\begin{align*}
\hbox{dist}\left(u, \mathcal{Z}_{0}\right)\leq C(n, p, \lambda)\delta(u).
\end{align*}
\end{theorem}

\medskip

\medskip

Before we proceed further let us briefly illustrate the necessity of the hypothesis {\bf (H1)}. Precisely why the assumptions on the dimension and  $\lambda$ in Theorem \ref{maintheorem-1} are being imposed in the critical case. Thanks to \eqref{S-inq}, the solutions to \eqref{eq1} are the critical points of 
\begin{align} \label{ilambda}
I_{\lambda}(u)= \frac{1}{2}\int \limits_{\B^{n}}\left(|\nabla_{\B^{n}}u|^{2}-\lambda u^{2}\right)~\dvg-\frac{1}{p+1} \int \limits_{\B^{n}}|u|^{p+1}~\dvg
\end{align}
defined on $H^1(\bn)$ and hence variational techniques apply to the energy functional $I_\lambda$. Therefore, one needs to know whether the following Palais-Smale condition holds: if for a sequence $\{u_k\}$ in $H^1(\bn),$ the energy $I_{\lambda}(u_k) \rightarrow d$ and the differential of the energy functional $I_{\lambda}^{\prime}(u_k) \rightarrow 0$ in $H^{-1}(\bn),$ then $\{u_k\}$ is precompact in $H^1(\bn).$ The hyperbolic space being infinite volume we need to impose further assumptions on the sequence $\{u_k\}$ to be able to conclude its compactness. Indeed,  let $u \in H^{1}\left(\mathbb{B}^{n}\right)$ and 
$b_k \in \bn$ such that $b_k \rightarrow \infty$ in the hyperbolic distance and let $\tau_{b_k}$ be the hyperbolic translation given by \eqref{hyperbolictranslation} such that $\tau_{b_k}(0) = b_k.$ Define $u_k = u \circ \tau_{b_k},$
then $\|u_k\|_{H^{1}\left(\mathbb{B}^{n}\right)} = \|u\|_{H^{1}\left(\mathbb{B}^{n}\right)}$ but $u_k \rightharpoonup 0$ weakly in $H^1(\bn).$ This shows that the embedding $H^1(\bn) \hookrightarrow L^p(\bn)$ is not compact for any $p  \leq \frac{2n}{n-2}$. At the sub-critical case $1<p<\tstar - 1,$ thanks to the Relich compactness theorem the problem is only at infinity. Indeed the class of radially symmetric and decreasing functions of $H^1(\bn)$ compactly embedded in $L^p(\bn)$ for every $1<p<\tstar - 1$ (\cite[Theorem 3.1]{BS}). However, in the critical case $p = \tstar -1$ the problem is both locally and at infinity and hence one can not exclude the possibility of concentration. The seminal work of Br\'{e}zis and Nirenberg \cite{BN}  gave a positive answer to this question which asserts that in a smooth bounded domain of the Euclidean case, compactness holds if the sequence exhibits small energy, a result in the spirit of Aubin \cite{AYConj}. The same argument used in \cite{BN} leads to the sequence is precompact in $H^1$-topology, whose credit the authors give to F. Browder. In particular, for the Euclidean ball, $\Be^n$ and the corresponding energy functional 
\begin{align*}
I_{\mu, \Be^n}(u) = \frac{1}{2}\int \limits_{\Be^{n}}\left(|\nabla u|^{2}+\mu u^{2}\right){\rm dx}-\frac{1}{p+1} \int \limits_{\Be^{n}}|u|^{p+1} {\rm dx}, \quad \mu >0 \ \mbox{a constant}
\end{align*}
the compactness prevails below the energy level $\frac{1}{n}S(\rn)^{\frac{n}{2}},$ which is precisely the energy of Aubin-Talenti bubbles defined in \eqref{AT bubbles}, provided $\mu >0$. Moreover, inspired by the works of  Aubin on Yamabe conjecture \cite{AYConj} they showed that the mountain pass energy always stays strictly below $\frac{1}{n}S(\rn)^{\frac{n}{2}}$ provided $n \geq 4$ and $0<\mu< \lambda_1(\Be^n),$ where $\lambda_1(\Be^n)$ is the first eigenvalue of the Dirichlet Laplacian on $\Be^n.$
In dimension $3,$ one needs to further impose $\mu > \frac{1}{4}\lambda_1(\Be^n).$ On the other hand, Pohozaev's identity shows that such statements are false for $\mu \leq 0$.

In addition, in a general smooth domain $\Omega$ the result of Brezis-Nirenberg \cite{BN} shows that for dimension $n \geq 4$ and $0 < \lambda < \lambda_1(\Omega)$
 the geometry of the domain playing no role
whatsoever for existence of solutions. On the other hand, Druet \cite{Druet} and Druet-Laurain \cite{DrLa} showed that the geometry
plays a role in dimension $n = 3,$  by proving non-existence of solutions when $\Omega$ is
star-shaped and $\lambda >0$ small (with no
a priori energy bound). Another point of view is that for $n = 3,$ the nonexistence
of solutions persists under small perturbations, but it does not for $n \geq 4$, in other words, the
Pohozaev obstruction is stable only for $n = 3.$

 The mentioned compactness result only gives a crude description of the Palais-Smale decomposition of the sequence. The subsequent seminal work of Struwe \cite{MS2} provides the complete description: For a sequence $u_{k} \in H_0^1(\Be^n)$ if $I_{\mu, \Be^n}(u_k) \rightarrow d, I_{\mu, \Be^n}^{\prime}(u_k) \rightarrow 0$ in $H^{-1}(\Be^n)$ then up to a subsequence $u_k$ converges in $H^1$-topology to a sum of Aubin-Talenti bubbles and a solution to the corresponding Euler-Lagrange equation.


\medskip

\noindent
\subsection{Conformal Change of Metric.}
Let us now translate the result of $\Be^n$ in the hyperbolic space $\bn$. We know that the conformal Laplacian or the Yamabe operator $P_{1,\bn} := -\Delta_{\bn} + \frac{(n-2)}{4(n-1)}R_{\bn} = -\Delta_{\bn} - \frac{n(n-2)}{4}$ is the first order conformal invariant operator where $R_{\bn}:= -n(n-1)$ is the scalar curvature  with respect to the metric $g_{\bn}$. It means that if  
$\tilde g = e^{2\psi}g$ is a conformal metric then $P_{1,\tilde g} (u) = e^{-(\frac{n}{2} +1)\psi} P_{1,\bn}(e^{(\frac{n}{2}-1)\psi}u)$ for every smooth function $u.$
The Poincar\'e metric being conformal to the Euclidean metric with $\psi (x)= \ln \left(\frac{1-|x|^2}{2}\right)$ we can 
transform \eqref{eq1} in the Euclidean space as follows:
 Let $u$ be a solution to \eqref{eq1}  with $p = \tstar -1.$  Set $\varphi:= e^{-(\frac{n}{2}-1)\psi}=\left(\frac{2}{1-|x|^{2}}\right)^{\frac{n-2}{2}}$, then $v :=\varphi u$ solves
 \begin{align}\label{eq2}
 -\Delta v- a(x)v=|v|^{\tstar-2}v, \ \ \ v \in H^{1}_{0}(\Be^{n}),
 \end{align} 
 where $a(x) = \frac{4\lambda-n(n-2)}{\left(1-|x|^{2}\right)^{2}}$ and $H^{1}_{0}(\Be^{n})$ is the Sobolev space on $\Be^n$ characterized by zero trace on the boundary $\partial\Be^n.$  Note that $a(x) > 0$ in $\Be^n$ whenever $\lambda > \frac{n(n-2)}{4}.$
 
 \medskip

 The equation \eqref{eq2} falls right under the framework of the works of Br\'{e}zis and Nirenberg \cite{BN}, subject to the delicate issue of singularity at the boundary in lower-order term $ a(x)v$. A priori there is no reason to believe that \eqref{eq2} would satisfy the Palais-Smale criteria. Should the statement like Theorem \ref{maintheorem-1} hold one must first ask if a sequence $\{u_k\}$  almost optimizes \eqref{S-inq} (i.e., $\delta(u_k) \rightarrow 0$) then is it possible that the sequence may contain Aubin-Talenti bubbles?
 Since the Aubin-Talenti bubbles have energy $\frac{1}{n}S(\rn)^{\frac{n}{2}}$, to exclude the possibility of Aubin-Talenti bubbles we need to ensure that the almost optimizing sequence $\{u_k\}$ contains lower energy compared to the Aubin-Talenti bubbles. It turns out this is indeed the case for $n \geq 4,$ and $\lambda > \frac{n(n-2)}{4}$ and the corresponding best Poincar\'e-Sobolev constant $S_{n, p, \lambda}(\B^n), \, p=\tstar-1$ is strictly less than $S(\rn)$ \cite{TT}. For the convenience of the readers, we recollect a proof of this result in Proposition~\ref{PEUCS} which follows the ideas of \cite{TT} and is based on testing a localized Aubin-Talenti bubble 
against \eqref{S-inq}. On the other hand, note that in dimension $3,$ the best Poincar\'e-Sobolev constant and Sobolev constant coincide \cite{BFLo}. Indeed, in this case $\lambda_1(\Be^3) = \pi^2$ and $\lambda_1(\mathbb{B}^3) = 1$ and hence one can not exclude the possibility of Aubin-Talenti bubbles in almost optimizing sequence.

\medskip

Next, we turn our attention to the stability from the Euler-Lagrange equation point of view as in \cite{CFM}. Recall that the solutions of \eqref{eq1} can be obtained as the critical points of $I_{\lambda},$ or in other words, $u_0$ solves \eqref{eq1} if and only if $I_{\lambda}^{\prime}(u_0) = 0$ in $H^{-1}(\B^n).$ Recall that we assume the optimizer $\mathcal{U}$ obtained in Theorem \ref{MSthm} satisfies $\|\calu\|_{\lambda}^2 = S_{n,p,\lambda}(\bn)^{\frac{p+1}{p-1}}.$ As a result $\calu$ has the energy $I_{\lambda}(\mathcal{U}) = \frac{p-1}{2(p+1)}S_{n,p,\lambda}(\B^n)^{\frac{p+1}{p-1}}.$
As in the seminal work of Struwe \cite{MS2}, in the hyperbolic setting Bhakta and Sandeep \cite{BS} have proved a non-quantitative version of the stability of the Euler-Lagrange equation \eqref{eq1}:

\begin{theorema}[Bhakta-Sandeep \cite{BS}]\label{t:BS1}
Let $n \geq 3$ and $\lambda,p$ satisfies the hypothesis {\bf (H1)}. Let $(u_{\alpha})$ be a Palais-Smale sequence of non-negative functions for $I_{\lambda}$ such that $\lim \limits_{\alpha\to\infty}I_{\lambda}(u_{\alpha})=\frac{p-1}{2(p+1)}S_{n,p,\lambda}(\B^n)^{\frac{p+1}{p-1}}$, then there exists a sequence of points $(\xi_{\alpha})$ in $\B^{n}$ such that as $\alpha \to \infty$
$$ u_{\alpha}=\calu\circ \tau_{\xi_{\alpha}}+o(1)~\hbox{ in } H^{1}(\B^{n}).$$
\end{theorema}

The above theorem directly follows from Proposition~\ref{PEUCS} (see Section 3) and the Palais-Smale decomposition results in \cite[Theorem 3.3]{BS}.  

\begin{remark}
{\rm
It is worth noticing from \cite[Theorem 3.3]{BS} that the profile decomposition of the Palais-Smale sequence at the arbitrary energy level is essentially a superposition 
of two types of sequences $\calu \circ \tau_{\xi},$ which we call \it hyperbolic bubble \rm and a \it localized Aubin-Talenti bubble. \rm We shall prove in 
Proposition~\ref{PEUCS} that for $p=\tstar-1$ and at the energy level specified as in Theorem~\ref{t:BS1}, the Aubin-Talenti bubbles are absent in their Palais-Smale decomposition.}
\end{remark}

\medskip

Our next result concerned the quantitative version of Theorem \ref{t:BS1} in the spirit of the results of Ciraolo-Figalli-Maggi  \cite{CFM}. Namely, if a non-negative function $u \in H^1(\B^n)$ almost solves equation \eqref{eq1} then $u$ is close to the $\mathcal{U} \circ \tau_{\xi}$ for some $\xi \in \B^n$ in a quantitative way.
\begin{theorem}\label{maintheorem-2}
Let $n \geq 3$  and $\lambda,p$ satisfies the hypothesis {\bf (H1)}. Let $(u_{\alpha})$ be a Palais-Smale sequence of non-negative functions for $I_{\lambda}$ such that $\lim \limits_{\alpha\to\infty}I_{\lambda}(u_{\alpha})=\frac{p-1}{2(p+1)}S_{n,p,\lambda}(\B^n)^{\frac{p+1}{p-1}}$, then there exists a  constant $C(n, p,\lambda)>0$ such that 
\begin{align*}
\hbox{dist}\left(u_{\alpha}, \mathcal{Z}\right)\leq C(n, p, \lambda) \|I'_{\lambda}(u_{\alpha})\|_{H^{-1}(\B^{n})}
\end{align*}
where $\mathcal{Z}$ denotes the manifold of non-negative solutions of \eqref{eq1}
\begin{align*}
\mathcal{Z}:=\{\mathcal{U}\circ\tau_{\xi}: \xi \in \B^{n}\}.
\end{align*}
and $\hbox{dist}\left(\cdot, \mathcal{Z}\right) = \inf_{\xi \in \bn}\| \cdot \ - \ \calu \circ \tau_{\xi}\|_{\lambda}.$ 
\end{theorem}

A direct application of Theorem \ref{maintheorem-2} is the following 

\begin{corollary}\label{corthm2}
Let $n\geq 3, \lambda$ and $p$ satisfies the hypothesis {\bf (H1)}. Then there exists $\epsilon_0 >0$  and a constant $C(n,p,\lambda)>0$ such that the following statement holds: for any non-negative function $u \in H^1(\bn)$ satisfying
\begin{align*}
(1-\epsilon_0)S_{n,p,\lambda}(\B^n)^{\frac{p+1}{p-1}} \leq \|u\|_{\lambda}^2 \leq (1+\epsilon_0)S_{n,p,\lambda}(\B^n)^{\frac{p+1}{p-1}}
\end{align*}
there holds
\begin{align*}
\hbox{dist}\left(u, \mathcal{Z}\right)\leq C(n, p, \lambda) \|I'_{\lambda}(u)\|_{H^{-1}(\B^{n})}.
\end{align*}
\end{corollary}

One can study the problem of Theorem \ref{maintheorem-2} relaxing non-negativity assumption but locally around one bubble as in \cite{CFM}. The proof is almost the same with necessary modifications.
In a forthcoming article \cite{BGKM}, we plan to address the local quantitative stability of \eqref{eq1} for more than one bubble on the hyperbolic space.

\medskip

We have utilized the above stability results in two applications: 

\begin{enumerate}

\item  Rate of stabilization and finite time extinction of the fast diffusion equations in the sub-critical regime, see Section~\ref{fd sec} for complete details. We want to stress that this is not a straightforward application of our stability results, as in \cite[Theorem~1.3]{CFM} and \cite[Theorem~5.1]{FG}. We postpone this discussion to subsection~1.1.

\medskip

\item  We obtain stability of the Hardy-Sobolev-Maz'ya inequality and the corresponding Euler-Lagrange equation on the regime of non-negative, partially symmetric functions.  We will provide the exact statement and the necessary background on this topic in Section~\ref{hsm sec} of this article. 
\end{enumerate}

\subsection{Fast Diffusion Equations: Main Hurdles and Novelty of our proof} In this subsection we briefly explain the novelty of our results and techniques to obtain a 
quantitative convergence rate for relative error for the sharp asymptotics of the solutions of the fast diffusion flows on the hyperbolic space.  
This problem was first considered and addressed in the hyperbolic space by Grillo-Muratori in \cite{GM}. The authors established $L^{\infty}$-convergence of the relative error for solutions 
of the sub-critical fast diffusion equations corresponding to radial initial data. Among other results, the authors  proved that near the extinction time the behaviour of the solutions is given by an appropriately defined profile 
$\calu_{m}$ (see \eqref{calum}). In addition, the relative error $\frac{u(., t)}{\calu_m(\cdot,t)}-1$ converges uniformly to 0: 

\begin{align*}
\lim_{t \rightarrow T-} \left \| \frac{u(\cdot, t)}{\calu_m(\cdot,t)} - 1\right\|_{L^{\infty}(\bn)} = 0,
\end{align*}
where  $T$ is the extinction time (see Section~\ref{fd sec} for the definition).  To date, the rate of convergence was not known and the question is quite elusive in the hyperbolic space setting. On the other hand,
Ciraolo-Figalli-Maggi \cite{CFM} and  Figalli-Glaudo \cite{FG}  exhibit a methodology to derive the rate of convergence of the relative error for the fast diffusion equation corresponding to the Yamabe-flow using their quantitative stability results of Sobolev inequality in the Euclidean space. The results of \cite{CFM, FG} preceded by Del-Pino and S\'{a}ze back in 2001 \cite{FD8} where the authors obtained stabilization of the solutions without a rate of convergence. The formal argument of \cite{CFM, FG} is as follows: in the logarithmic time scale, the rescaled fast diffusion flow decays along certain energy functional which corresponds to the Euler-Lagrange functional of the Sobolev embedding. It turns out that the dissipation of the energy along the flow can be estimated in terms of the $H^{-1}$-norm of the deficit functional; hence, their stability results are in force.
As a consequence, the energy of the flow stabilizes to the energy of the Aubin-Talenti bubble exponentially fast,
and one can estimate the $L^{\tstar}$-distance of the flow from the manifold of Aubin-Talenti bubbles in terms of the energy deficit. In particular,  the corresponding relative linear entropy of the error decays exponentially fast along the flow (see Section \ref{fd sec}).
The next step is to bridge between the relative linear entropy of the error (or the $L^{\tstar}$-norm of the error) to the $L^{\infty}$-norm of the relative error. This bridge has already been constructed by Del-Pino and S\'{a}ze \cite{FD8} by lifting the equation on the sphere via stereographic projection. One important fact that was crucially used is the fact that the Jacobian of the stereographic projection is a constant multiple of $U[0,1]^{\tstar}$ where $U[0,1]$ is an Aubin-Talenti bubble.
Subsequently, the uniform decay of the relative error can be proved by interpolation inequalities and the gradient estimates obtained 
in \cite[Proposition~5.1]{FD8}. However, there are several hidden and challenging difficulties if one wants to implement this idea of Ciraolo-Figalli-Maggi and 
 Figalli-Glaudo in finding the convergence rate for the sub-critical fast diffusion flows on the hyperbolic space. When posed on the hyperbolic space, the bridge from the relative linear entropy of the error (or the $L^{p+1}$-norm of the error) to the uniform bound of the relative error is quite involved and that is where the main difficulty appears. On the other hand, this bridge from linear entropy of the error to uniform relative error estimate has appeared earlier in the literature in other contexts of the fast diffusion flow in the Euclidean space. See for example, \cite{FD1,FD3} and the very recent treatise by Bonforte and Figalli \cite{FD7}.
 \medskip
 
 We present below a rough sketch of ideas involved in the proof. Our strategy is two-pronged: 
 
 \begin{itemize}
 
 \item In the first step, following the idea of Ciraolo-Figalli-Maggi and Figalli-Glaudo we establish an exponential decay of $\| u \, - \calu_m  \|_{L^{p+1}}$ using our quantitative stability result of the Poincar\'e-Sobolev inequality and subsequently, we could derive exponential decay estimates on the relative linear entropy, see Section~\ref{fd sec} for more details. 
 
 \medskip
 
 \item Secondly, to get the uniform bound on the relative error in terms of the linear entropy, we partly follow and adapt the ideas of Bonforte-Figalli \cite{FD4}. In particular, we derive smoothing estimates for the relative error in terms of linear entropy. We want to point out that because of the non-compactness of the hyperbolic space and the specific nature of the volume element, deriving such smoothing estimates requires a novel way of analyzing the integrals involving the error, which makes the proofs of smoothing estimates very delicate and involved. We perform boundary and interior estimates separately to achieve our desired goal. Unfortunately at this step, we arrive at a dimensional restriction (partly because of the integrability of $\calu^{p-1}$). We prove the smoothing estimates for the relative error holds true when $2< p < {\tstar} -1,$ i.e., $n \in [3, 5]$ (see Theorem~\ref{linftyentropy}). Combining the
  decay estimates on linear entropy Theorem~\ref{entropy decay} and Theorem~\ref{linftyentropy}, we provide the desired rate of convergence of the relative error in uniform topology at the extinction time for $2< p < {\tstar} -1$ (see Theorem~\ref{asymptotic}). As of now, we don't know whether the assumption on $p$ is purely technical or a generic one. Because of the sharp decay of $\calu,$ we infer that $\calu^{p-1} \in L^1(\bn)$ if and only if $p>2,$ which is strongly reflected in our argument. We refer Section~\ref{relative-error-theorem} for more details.

 \end{itemize}

\medskip

The organization of the rest of the paper is as follows: In Section \ref{pre} we recall some basics of the hyperbolic space. In Section \ref{key lem} we prove all the intermediate lemmas that will be used to prove our main theorems. We present the proof of Theorem \ref{maintheorem-1} and Theorem \ref{maintheorem-2} in Section \ref{proof1} and \ref{proof2} respectively. In Section~\ref{fd sec}, we formulate the necessary background and notations related to the subcritical fast diffusion flows in the hyperbolic space and derive the decay of linear entropy with rate.
In Section~\ref{relative-error-theorem}, we establish the smoothing estimates for the relative error in terms of linear entropy and prove our main result Theorem~\ref{asymptotic} concerning the fast diffusion flow. Finally, in Section~\ref{hsm sec}, as another application of our stability results, we prove the stability of the Hardy-Sobolev-Maz'ya inequality. 



\section{Geometric and functional analytic Preliminaries}\label{pre}

The Euclidean unit ball $\Be^n:= \{x \in \mathbb{R}^n: |x|^2<1\}$ equipped with the Riemannian metric
\begin{align*}
{\rm d}s^2 = \left(\frac{2}{1-|x|^2}\right)^2 \, {\rm d}x^2
\end{align*}
constitute the ball model for the hyperbolic $n$-space, where ${\rm d}x$ is the standard Euclidean metric and $|x|^2 = \sum_{i=1}^nx_i^2$ is the standard Euclidean length. By definition, the hyperbolic $n$-space is a $n$-dimensional complete, non-compact Riemannian manifold having constant sectional curvature equal to $-1$ and any two manifolds sharing the above properties are isometric \cite{RAT}. In this article, all our computations, with an exception of Section \ref{hsm sec}, will involve only the ball model and will be denoted by $\bn$. We denote the inner product on the tangent space of $\mathbb{B}^n$ by $\langle \cdot, \cdot \rangle_{\bn}$ and
the volume element is given by $\dvg = \left(\frac{2}{1 - |x|^2}\right)^n {\rm d}x,$ where ${\rm d}x$ denotes the 
Lebesgue measure on $\mathbb{R}^n.$\\
Let $\nabla_{\bn}$ and $\Delta_{\bn}$ denote gradient vector field
and Laplace-Beltrami operator respectively.  Therefore, in terms of local (global) coordinates $\nabla_{\bn}$ and $\Delta_{\bn}$ takes the form:
\begin{align*} 
 \nabla_{\bn} = \left(\frac{1 - |x|^2}{2}\right)^2\nabla,  \quad 
 \Delta_{\bn} = \left(\frac{1 - |x|^2}{2}\right)^2 \Delta + (n - 2)\left(\frac{1 - |x|^2}{2}\right)  x \cdot \nabla,
\end{align*}
where $\nabla, \Delta$ are the standard Euclidean gradient vector filed and Laplace operator respectively, and `$\cdot$' denotes the 
standard inner product in $\mathbb{R}^n.$

\medskip 

\noindent
 {\bf Hyperbolic distance on $\bn.$} The hyperbolic distance between two points $x$ and $y$ in $\bn$ will be denoted by $d(x, y).$ The hyperbolic distance between
$x$ and the origin can be computed explicitly  
\begin{align*}
\rho := \, d(x, 0) = \int_{0}^{|x|} \frac{2}{1 - s^2} \, {\rm d}s \, = \, \log \frac{1 + |x|}{1 - |x|},
\end{align*}
and therefore  $|x| = \tanh \frac{\rho}{2}.$ Moreover, the hyperbolic distance between $x, y \in \bn$ is given by 
\begin{align*}
d(x, y) = \cosh^{-1} \left( 1 + \dfrac{2|x - y|^2}{(1 - |x|^2)(1 - |y|^2)} \right).
\end{align*}
As a result, a subset of $\bn$ is a hyperbolic sphere in $\bn$ if and only if it is a Euclidean sphere in $\mathbb{R}^n$ and contained in $\bn$ possibly 
with a different centre and different radius which can be explicitly computed from the formula of $d(x,y)$ \cite{RAT}.  Geodesic balls in $\bn$ of radius $r$ centred at $x \in \bn$ will be denoted by 
$$
B_r(x) : = \{ y \in \bn : d(x, y) < r \}.
$$

Next, we introduce the concept of hyperbolic translation.
\subsection{Hyperbolic translation}
\begin{definition}[Hyperbolic translation]\label{hypt}
 For $b \in \mathbb{B}^n,$ we define the hyperbolic translation $\tau_{b}: \mathbb{B}^n \rightarrow \mathbb{B}^n$ by 
  \begin{align} \label{hyperbolictranslation}
  \tau_{b}(x) := \frac{(1 - |b|^2)x + (|x|^2 + 2 x \cdot b + 1)b}{|b|^2|x|^2 + 2 x\cdot b  + 1}.
 \end{align}
\end{definition}
 Then $\tau_{b}: \mathbb{B}^n\rightarrow \mathbb{B}^n$ is an isometry, (see \cite{RAT}, theorem 4.4.6) for details and
 further discussions on isometries.
  
For the convenience of the reader, we list several well-known properties of the hyperbolic translation in the next lemma. The proof follows by now standard properties of the M\"{o}bius transformations on the ball model which we skip for brevity and refer to the book \cite{Stoll}. 

\begin{lemma}\label{lemma1}
  For $b \in \B^n,$ let $\tau_b$ be the hyperbolic translation of $\mathbb{B}^n$ by $b.$ Then for every $u \in C_c^{\infty}(\mathbb{B}^n),$ there holds,

\begin{itemize}
  
\item[(i)] $\Delta_{\bn} (u \circ \tau_b) = (\Delta_{\bn} u) \circ \tau_b, \ \ \langle \nabla_{\bn} (u \circ \tau_b),
   \nabla_{\bn} (u \circ \tau_b)\rangle_{\bn} = \langle(\nabla_{\bn} u) \circ \tau_b, (\nabla_{\bn} u) \circ \tau_b\rangle_{\bn}.$
  
  \item[(ii)] For every open subset $U$ of $\mathbb{B}^n$
  \begin{align*}
  \int_{U} |u \circ \tau_b|^p \, \dvg = \int_{\tau_b(U)} |u|^p \, \dvg , \ \mbox{for all} \ 1 \leq p < \infty.
 \end{align*}
 
 \item[(iii)] For every $\phi, \psi \in C_c^{\infty}(\B^n)$,
 \begin{align*}
 \int_{\B^n} \phi(x) (\psi \circ \tau_b)(x) \, \dvg = \int_{\B^n} (\phi \circ \tau_{-b})(x) \psi(x) \, \dvg. 
 \end{align*}
\end{itemize}
\end{lemma}
By density, the above formulas hold as long as the integrals involved are finite.

\subsection{The Sobolev space $H^1(\mathbb{B}^n)$:} Thanks to the Poincar\'{e} inequality, we define the Sobolev space on $\bn$, denoted by $H^1(\mathbb{B}^n)$, as the completion of $C_c^\infty(\mathbb{B}^n)$ with respect to the norm
\begin{align*}
 \|u\|_{H^1(\mathbb{B}^n)} :=  \left(\int_{\mathbb{B}^n} 
 |\nabla_{\bn} u|^2  \, \dvg \right)^{\frac{1}{2}},
\end{align*}
where  $|\nabla_{\bn} u| $ is given by
 \begin{align*}
  |\nabla_{\bn} u| := \langle \nabla_{\bn} u, \nabla_{\bn} u \rangle^{\frac{1}{2}}_{\bn}.
\end{align*}
Using Poincar\'e inequality once again, $\|u\|_{\lambda}$ defined by 
\begin{align*}
\|u\|_{\lambda}^2 : = \int \limits_{\B^{n}}\left(|\nabla_{\bn} u|^{2}-\lambda u^{2}\right) \dvg~\hbox{ for all } u\in H(\B^{n}),
\end{align*}
is an equivalent $H^1(\mathbb{B}^n)$ norm for $\lambda < \frac{(n-1)^2}{4}.$ However, for $\lambda = \frac{(n-1)^2}{4},$ $\|u\|_{\lambda}$ is not an equivalent norm because, in general, $u$ may fail to be square integrable. Moreover, thanks to Lemma \ref{lemma1}, the norms $\|u\|_{H^1(\bn)}, \|u\|_{\lambda}$ are invariant under the action of isometries of the ball model. 

\medskip 

We also define the subset of $H^1(\bn)$ consisting of radial functions
\begin{align*}
H^1_{\mbox{\tiny{rad}}}(\bn) := \{u \ \mbox{radial} \ | \ \int_{0}^{\infty} \left[u^2(\rho) + (u^{\prime}(\rho))^2\right] (\sinh \rho)^{n-1} \ d\rho < \infty \},
\end{align*}
where $\prime$ denotes the differentiation with respect $\rho.$


\section{Non-degeneracy, spectral gap and Eigen-space of the Linearised operator}\label{key lem}

We recall some of the well-known results concerning the  Poincar\'e-Sobolev equations on the hyperbolic space. Mancini and Sandeep \cite{MS} studied the existence and uniqueness of finite energy positive solutions of 
\begin{equation}\label{hsm}
-\Delta_{\mathbb{B}^{n}} u \, - \, \lambda \, u \, = \, u^{p}, \quad u \in H^{1}\left(\mathbb{B}^{n}\right),
\end{equation}
under the assumptions $\lambda \leq \frac{(n-1)^2}{4}$ and $1 < p \leq \frac{n+2}{n-2}$ if $n \geq 3,$  $1 < p < \infty$ if $n=2.$ According to their results, in the subcritical 
case, i.e., $1 < p < 2^{\star} -1$ if $n \geq 3,$ and any $p > 1$ is allowed if $n =2,$ the problem \eqref{hsm} has a positive solution if and only if $\lambda < \frac{(n-1)^2}{4}.$ On the other hand in the critical case i.e., $p = \tstar - 1$ a positive finite energy solution exists if and only if $n \geq 4$ and $\frac{n(n-2)}{4} < \lambda < \frac{(n-1)^2}{4}.$
These 
positive solutions are also shown to be unique up to hyperbolic isometries except possibly when $n =2$ and $\lambda > \frac{2(p-1)}{(p+3)^2}.$  Subsequently, the existence of sign-changing solutions, compactness and non-degeneracy were studied in \cite{BS, DG1, DG2}.
We remark that the existence and multiplicity of infinite energy solutions of \eqref{hsm} for critical and as well as for super-critical exponents also have been studied in the literature \cite{BFG, BGGV}.

\medskip

For our convenience, we list some identities and assumptions used in this article.

\begin{itemize}
\item[$\bullet$] We restrict ourselves to $n \geq 3$ in this article and we assume $1 < p \leq \tstar -1,$ where $\tstar = \frac{2n}{n-2}$. When $p = \tstar -1$ the following identities hold
\begin{align*}
\frac{p+1}{p-1} = \frac{n}{2}, \ \ \frac{1}{2} - \frac{1}{p+1} = \frac{2}{n}.
\end{align*}

\item[$\bullet$]  For $n  = 3,$ we exclude the critical case, namely, 
we assume $1 < p < \tstar - 1$ and $\lambda < \frac{(n-1)^2}{4}.$

\item[$\bullet$] For $n \geq 4,$ the critical case is also allowed $1 < p \leq \tstar -1$ but the restriction on $\lambda \in (\frac{n(n-2)}{4}, \frac{(n-1)^2}{4})$ is assumed.

\item[$\bullet$] For a positive function $w$ defined on $\bn,$ we define the weighted $L^2$-space
\begin{align*}
L_w^2(\bn) : = \left\{u: \bn \to \mathbb{R} \ measurable \ | \ \int_{\bn}|u|^2w \ \dvg < \infty \right\},
\end{align*}
and the associated inner product $\langle u, v \rangle_{L^2_w(\bn)} = \int_{\bn} uvw \ \dvg.$
\end{itemize}
 Under the above assumption on $n,p,\lambda$ we know that there exists a positive, unique (up to hyperbolic isometries and multiplication by constant) optimizer of \eqref{S-inq} denoted by $\calu$ and satisfies the normalization 
\begin{align*}
\|\calu\|_{\lambda}^2 = \int_{\bn} \calu^{p+1} \ \dvg = S_{n,p,\lambda}(\bn)^{\frac{p+1}{p-1}}.
\end{align*}
 and solves 
\begin{align}\label{extremal-critical}
-\Delta_{\bn} \mathcal{U} - \lambda \,  \mathcal{U} =  \mathcal{U}^{p}.
\end{align}

\medskip 

\subsection{Nondegeneracy.}
In this subsection, we collect a few key lemmas needed for the proof related to  the eigenvalues and eigenfunctions of the linearized operator 
$(-\Delta_{\bn} - \lambda)/\mathcal{U}^{p-1}$.
We know that if $\tau_b$ is a hyperbolic translation then $\mathcal{U} \circ \tau_b,$ also solves \eqref{extremal-critical}. This fact indicates that the solution $\mathcal{U}$ must degenerate or, in other words, the kernel of the linearized operator contains non-trivial elements. It was shown in \cite{DG2} that the degeneracy happens only along an $n$-dimensional subspace characterized by the vector fields 
\begin{align*}
V_i(x) := (1 + |x|^2) \frac{\partial}{\partial x_i} \, - \, 2x_i \sum_{j =1}^{n} x_j \, \frac{\partial}{\partial x_j}.
\end{align*}
 for $i = 1, \ldots, n.$ More precisely, we define 
\begin{align*}
 \Phi_i(x) := \frac{d}{dt} \Big|_{t=0} \mathcal{U} \circ \tau_{t e_i} , \ \ 1 \leq i \leq n,
 \end{align*}
then $\Phi_{i}(x) = V_i(\calu),$ $\Phi_i$ solves the eigenvalue problem 
\begin{align} \label{eigenvalue}
-\Delta_{\bn} \Phi_i - \lambda \Phi_i = p \,  \mathcal{U}^{p-1} \, \Phi_i,
\end{align}
and the degeneracy in the solution space to \eqref{eigenvalue} can occur only along the directions $\Phi_i, 1 \leq i \leq n.$ 
 \begin{theorema}[\cite{DG2}]
 Let $V_i$ be the vector fields in $\bn$ defined above and $\Phi_i = V_i(\calu).$ Then $\{\Phi_i\}_{i=1}^n$ forms a basis for the kernel of $(-\Delta_{\bn} - \lambda - p\calu^{p-1}).$
 \end{theorema}

As a result, we obtain complete  information on the first and the second eigenvalues and corresponding eigenspaces of the operator 
$$\mathcal{L} := (-\Delta_{\bn} - \lambda)/\mathcal{U}^{p-1}.$$

\subsection{Spectral gap and eigenspaces of the linearized operator} We prove the following

\begin{proposition} \label{eigen value lemma}
The first and the second eigenvalues of the operator $$\mathcal{L} := (-\Delta_{\bn} - \lambda)/\mathcal{U}^{p - 1}$$ 
are respectively $1$ and $p.$ Moreover, the first eigenspace is one dimensional and spanned by $\mathcal{U}$ and the second eigenspace is $n$-dimensional and sppaned by $\{\Phi_i\}_{1 \leq i \leq n}.$
\end{proposition}

We divide the proof of Proposition~\ref{eigen value lemma} into several lemmas. We set $w = \calu^{p-1}$
and recall the corresponding weighted $L^2$ space by 
\begin{align*}
L^2_{\calu^{p-1}}(\bn):= \{u \ | \ \int_{\bn} |u|^2\calu^{p-1} \ \dvg < \infty\}.
\end{align*}

 We begin with a lemma concerning the compactness of an embedding. 

\begin{lemma} \label{compact embedding}
The embedding $H^1(\B^n) \hookrightarrow L^2_{\calu^{p-1}}(\bn)$ is compact.
\end{lemma}

\begin{proof}

It follows from \cite{MS} that  $$\mathcal{U} (x) \rightarrow 0,\  \mbox{as} \ d(x, 0) \rightarrow \infty,$$ and $\calu \in L^{\infty}(\bn).$ Therefore using \cite[Proposition~2.2]{DG2} the proof 
follows. 

\end{proof}

\begin{lemma} \label{compact self adjoint operator}
Let $n \geq 3$ and  $\lambda, p$ satisfies the hypothesis {\bf(H1)}, then $\mathcal{L}^{-1}$ is well-defined, compact, self-adjoint and positive operator  on $L^2_{\calu^{p-1}}(\bn)$. In particular, there exists a countable orthonormal basis $\{\Psi_i\}_{i=1}^{\infty}$ of $L^2_{\calu^{p-1}}(\bn)$ consisting of eigen-vectors of $\mathcal{L}^{-1}$ with corresponding eigenvalues $\{\eta_i\}_{i=1}^{\infty}$ such that $\eta_i \searrow 0$ as $i \rightarrow \infty.$

\end{lemma}

\begin{proof}

We first claim that $\mathcal{L}^{-1} : L^2_{\calu^{p-1}}(\bn) \rightarrow L^2_{\calu^{p-1}}(\bn)$ is a bounded, linear, 
compact operator. The proof follows from the Poincar\'e-Sobolev inequality, Lemma \ref{compact embedding} and the identity $\frac{1}{2} + \frac{p-1}{2(p+1)} + \frac{1}{p+1} = 1.$ 
Indeed, for every $f \in L^2_{\calu^{p-1}}(\bn)$ and $\psi \in H^1(\B^n),$ using H\"older inequality we have
\begin{align} \label{p1}
\int_{\B^n} f \psi  \, \mathcal{U}^{p-1} \, \dvg &= \int_{\B^n} \mathcal{U}^{\frac{p-1}{2}}f  \cdot  \, \mathcal{U}^{\frac{p-1}{2}} \cdot \psi \, \dvg \notag\\
&\leq \left(\int_{\B^n} f^2 \,  \mathcal{U}^{p-1} \, \dvg \right)^{\frac{1}{2}} \left(\int_{\B^n} \,  \mathcal{U}^{p+1} \, \dvg\right)^{\frac{p-1}{2(p+1)}} \left(\int_{\B^n} |\psi|^{p+1}\, \dvg \right)^{\frac{1}{p+1}}. \notag \\
&\leq \| f\|_{L^2_{\calu^{p-1}}(\bn)} \| \psi \|_{\lambda},
\end{align}
where we have used $\| \mathcal{U} \|_{\lambda}^2 = \| \mathcal{U} \|_{p+1}^{p+1} = S_{n,p,\lambda}(\B^n)^{\frac{p+1}{p-1}}.$
Hence for every $\lambda < \frac{(n-1)^2}{4}, \psi \mapsto \int_{\bn} f \psi \calu^{p-1}\ \dvg$ is an element of $H^{-1}(\B^n)$ 
 and by Riesz representation theorem, for every $f \in L^2_{\calu^{p-1}}(\bn),$ there exists a solution $\phi \in H^1(\B^n)$ to
\begin{align*}
-\Delta_{\bn} \phi - \lambda \phi =  \mathcal{U}^{p-1}f, \ \ \mbox{in} \ \B^n.
\end{align*}
Moreover, such solutions are unique and can be seen by applying \eqref{p1} to the solution $\phi$ 
\begin{align*}
\|\phi\|_{\lambda} \leq \|f\|_{L^2(\B^n, \, \mathcal{U}^{2^{\star}-2} \, \dvg)}.
\end{align*}
If we define the continuous map $L^2_{\calu^{p-1}}(\bn)  \mapsto H^1(\B^n) \subset L^2_{\calu^{p-1}}(\bn)$ by 
$T(f) = \phi$ then for any $\psi \in H^1(\B^n),$
\begin{align} \label{p2}
\int_{\B^n} f \psi  \, \mathcal{U}^{p-1} \ \dvg= \int_{\B^n} (\langle \phi, \psi \rangle_{\bn} - \lambda \phi \psi) \ \dvg=
\int_{\B^n} (-\Delta \phi - \lambda \phi) \psi \ \dvg.
\end{align}
As a result, $f \mathcal{U}^{p-1} = (-\Delta - \lambda) T(f)$ and hence $T = \mathcal{L}^{-1}.$

 By Lemma~\ref{compact embedding}, $\mathcal{L}^{-1}$ viewed as an operator on  $L^2_{\calu^{p-1}}(\bn)$ is compact  
and by symmetry in \eqref{p2},  it is self-adjoint. Hence the spectral theory of compact self-adjoint operators applies. One more application of the Poincar\'e-Sobolev inequality implies that for $\lambda < \frac{(n-1)^2}{4},$ 
\begin{align*}
 \langle T(f), f \rangle_{L^2_{\calu^{p-1}}(\bn)}=\int_{\B^n} f \phi \, \mathcal{U}^{p-1} \dvg \, = \, \int_{\B^n} |\nabla \phi|^2 \, \dvg \, - \,  \lambda \int_{\B^n} \phi^2 \, \dvg \geq 0,
\end{align*}
 and equality holds if and only if $\phi \equiv 0  \equiv f$ proving $\mathcal{L}^{-1}$ is positive.

\end{proof}

\noindent
{\bf Proof of Proposition~\ref{eigen value lemma}:}
\begin{proof}
Using  Lemma~\ref{compact self adjoint operator}, we conclude that all  the eigenvalues of $\mathcal{L}^{-1}$ are real, positive, and the corresponding eigenfunctions make up an 
orthonormal basis of $L^2_{\calu^{p-1}}(\bn).$ Also, observe that for $\eta \neq 0,$ we have $\mathcal{L}^{-1} v = \eta \, v$ if and only if
 $\mathcal{L}\, v = \mu \, v$ where $\mu = \frac{1}{\eta}.$ Hence there exists a non-decreasing sequence of positive eigenvalues $\mu_i$  and the corresponding eigenfunctions $ \Psi_i \in H^1(\bn)$ of $\mathcal{L}$ such that 
 the first eigenvalue is simple and 
 the eigenvalues are characterised by the Rayleigh quotient
$$
\mu_i \, = \, \inf_{\Psi \in H^1(\bn), \Psi \perp \{ \Psi_1,  \ldots \Psi_{i-1}\}} \dfrac{\|\Psi\|_{\lambda}^2} {\int_{\bn}  \Psi^2 \calu^{p-1} \, \dvg}.
$$  
where $\perp$ denotes the orthogonality in $L^2_{\calu^{p-1}}(\bn).$
 
 \medskip 
 
 Now we shall derive explicitly the first and the second eigenvalues. 
 
 \medskip 
 
{\bf First Eigenvalue:}   First we observe that  the equation 
\begin{align} \label{evp_mu}
- \De_{\bn} \Psi - \lambda \, \Psi \, = \, \mu \, \calu^{p-1} \, \Psi,
\end{align}
does not admit a non-trivial solution for $\mu < 1$. Indeed, if \eqref{evp_mu} admits a solution $\phi,$ then multiplying \eqref{evp_mu} by $\phi,$ integration by parts, H\"{o}lder and Poincar\'{e}-Sobolev inequality give
$$\|\phi\|_{\lambda}^2\leq\mu\left(\int_{\bn}\calu^{p+1}dv_{\bn}\right)^\frac{p-1}{p+1}\left(\int_{\bn}|\phi|^{p+1}dv_{\bn}\right)^\frac{2}{p+1}=\mu  S_{n, p, \lambda}(\bn)\left(\int_{\bn}|\phi|^{p+1}dv_{\bn}\right)^\frac{2}{p+1}.$$
Since $\mu < 1$, we get
$$\frac{\|\phi\|_{\lambda}^2}{\left(\int_{\bn}|\phi|^{p+1}dv_{\bn}\right)^\frac{2}{p+1}}<S_{n,p,\lambda}(\bn)$$
 leading to a contradiction to the definition of $S_{n,p, \lambda}(\bn)$. Hence the claim follows.   
Using the above observation and the fact that $\calu$ satisfies \eqref{extremal-critical} and the first eigenvalue is simple, we conclude that 
$$
\mu_1 = 1, 
$$
  and the corresponding eigenfunction $\Psi_1 = \calu.$
 
 \medskip 
 
 {\bf Second Eigenvalue :}  In \cite{DG2}, it was shown that all the solutions to 
 
\begin{equation*}
 -\De \Psi - \lambda \Psi \, = \, p\,  \calu^{p-1} \, \Psi
\end{equation*}
 
 are linear combination of  $\Phi_i(x) := \frac{d}{dt} \Big|_{t=0} \mathcal{U} \circ \tau_{t e_i} , \ \ 1 \leq i \leq n.$ To show $\mu_2 = p$ it is enough to prove that \eqref{evp_mu} does not admit any nontrivial solution for $1 < \mu < p.$
We denote the positive half of $\bn$ in the $x_1$-direction by $\B^n_{+} \, = \, \{x \in \bn : x_1 > 0 \}$  and consider the eigenvalue problem 
\begin{align}\label{evp_mu2}
 -\De \Psi - \lambda \Psi = \bar{\mu}  \, \calu^{p-1} \, \Psi \quad \mbox{in} \ \B^n_+, \ \Psi \in H^1_0(\bn_+),
 \end{align}
where $H_0^1(\bn_+)$ denotes the closure of $C_c^{\infty}(\bn_+)$ with respect to the norm $\|\cdot\|_{\lambda}.$ Taking into account the compact embedding $H^1_0(\B^n_+) \hookrightarrow L^2_{\calu^{p-1}}(\B^n_+),$ which follows from Lemma \ref{compact embedding}, and by similar reasoning as above we conclude that the operator $\left((-\Delta - \lambda)/\calu^{p-1}\right)^{-1}$ is a compact, self-adjoint operator on $L^2_{\calu^{p-1}}(\B^n_+).$ Moreover, the first eigenvalue is simple and the first eigen function does not change sign. $\calu$ being positive, all the other eigenfunctions must change the sign.

Since $\Phi_1 \in H^1_0(\bn_+),$ does not change sign in $\bn_+$ and satisfies  \eqref{evp_mu2}
with $\bar{\mu} = p$ we conclude, the first eigenvalue of \eqref{evp_mu2} is $\bar{\mu}_1 = p.$

 \medskip 
 
 Now let us assume that $\psi \in H^1(\B^n)$ is a solution to 
 
 \begin{equation}\label{eqrad}
 -\De \psi - \lambda \, \psi = \mu \, \calu^{p-1} \, \psi \quad \mbox{in} \ \bn \quad \mbox{with} \ \mu < p.
 \end{equation}
 Let us define  $\tilde{\psi}(x) := \psi (x_1, x_2, \ldots, x_n) \, - \, \psi (-x_1, x_2, \ldots, x_n),$ then clearly $\tilde{\psi} \in H^1_0(\B^n_+) $ and solves 
 \begin{align*}
 -\De \tilde{\psi_2}(x)  \, - \, \lambda \, \tilde{\psi_2}(x) = \mu \, \calu^{p-1} \, \tilde{\psi_2}(x) \quad \mbox{in} \ \B^n_+,
 \end{align*}
 where $\mu < p.$  This would contradict the definition of first eigenvalue in $\B^n_+$ unless $ \tilde{\psi}(x) \equiv 0,$ i.e., 
 $\psi$ is symmetric about the $x_1$- axis. Repeating this process in all directions we conclude that $\psi$ is radially symmetric about the origin.  It is proved in
\cite[Theorem~3.1]{DG2} that $\psi \, \equiv \, 0$ whenever $\psi \in H^1_{r} (\bn),$ and satisfies \eqref{eqrad}.
  
 This concludes $\mu_2 = p$ and by non-degeneracy result of \cite{DG2} the eigenfunctions corresponding to $\mu_2$ are linear span of $\Phi_i,$ $i = 1, \dots, n$. This completes the proof of the proposition. 
 \end{proof}


\medskip

\begin{remark}\label{r:4-6-1}
{\rm 
More generally one can consider the operator $(-\Delta_{\bn} - \lambda)/\mathcal{U}_{b}^{p-1},$ where $\mathcal{U}_b = \mathcal{U} \circ \tau_b.$  It turns out that the eigenvalues of the operator does not depend on  $b$ and hence without loss of generality we assume $b = 0.$ Indeed the eigenvalues of the operator $(-\Delta_{\bn} - \lambda)/\mathcal{U}_{b}^{p-1},$ are characterised by 
$$
\mu_i \, = \, \inf_{\Psi \in H^1(\bn), \Psi \perp \{ \Psi_1,  \ldots \Psi_{i-1}\}} \dfrac{\|\Psi\|_{\lambda}^2} {\int_{\bn}   \Psi^2 \,\calu_b^{p-1} \dvg},
$$
where $\Psi_i$ are the eigenfunctions of the operator $(-\Delta_{\bn} - \lambda)/\mathcal{U}_{b}^{p-1},$ corresponding to the eigenvalues $\mu_i$. Since $\phi \in H^1(\bn)$ if and only if 
$\phi_{-b} = \phi \circ \tau_{-b} \in H^{1}(\bn)$ and notice that the following two identities hold

$$
\int_{\bn} \left( |\nabla_{\bn} (\phi_{-b})|^2 \, - \, \lambda (\phi_{-b})^2 \right) \, \dvg \, = \, \int_{\bn} \left( |\nabla_{\bn} \phi|^2 \, - \, \lambda \phi^2 \right) \, \dvg,
$$
and 
$$
\int_{\bn} \calu_{b}^{p-1} \, \phi^2 \, \dvg = \int_{\bn} \calu^{p-1} \, (\phi_{-b})^2 \, \dvg.
$$
The above two identities immediately establish that the first eigenvalue does not depend on $b \in \bn.$ Using Lemma~\ref{lemma1}, we conclude 
$$
\calu \, \perp \, \Phi_i \quad \mbox{if and only if} \quad \calu_b \, \perp \, (\Phi_b)_i, 
$$
where $(\Phi_b)_i = \frac{d}{dt} |_{t =0} (\calu_b \circ \tau_{te_i}),$ and hence the second eigenvalue also does not depend on $b.$ 
}

\end{remark}
\medskip
\begin{remark}
{\rm 
Let us make a small remark about Bianchi-Egnell's result on the eigenvalues in the Euclidean case. In \cite{BE}, the authors considered the normalization $\|\nabla U\|_{L^2(\rn)} = 1$ for which one gets the first and the second eigenvalues  $S(\rn)^{\frac{\tstar}{2}}$ and $(\tstar - 1)S(\rn)^{\frac{\tstar}{2}}$ respectively. In our case, we used the normalization $\|\calu\|_{\lambda}^2 = \| \calu \|_{L^{p+1}(\bn)}^{p+1} = S_{n,p,\lambda}(\bn)^{\frac{p+1}{p-1}}.$ The later convention makes the notational complications relatively better when studying the stability from the Euler-Lagrange point of view. 
}
\end{remark}

\medskip

\subsection{Comparison of best Sobolev and Poincar\'{e}-Sobolev constants}
The next proposition validates the non-existence of the Aubin-Talenti bubbles in the critical case provided $n\geq 4$ and $\lambda$ is in {\it good range \rm}. We remark that for $n = 3$ the following lemma is not true \cite{BFLo} and for $n \geq 4$ we borrow the idea of \cite{TT}.

\begin{proposition}\label{PEUCS}
Let $n \geq 4,\, p = \tstar-1$ and $\frac{n(n-2)}{4}< \lambda<\frac{(n-1)^{2}}{4}$. Then $S_{n,p, \lambda}(\B^{n})<S(\R^{n})$, where $S(\R^{n})$ and $S_{n, p,\lambda}(\B^n)$ are the best Sobolev and Poincar\'e-Sobolev constants respectively.
\end{proposition}
\begin{proof}
For $p=\tstar-1$ we have
\begin{align*}
S_{n, p,\lambda}(\B^{n})=\inf \limits_{u \in C^{\infty}_{c}(\B^{n})\setminus\{0\}}\frac{\displaystyle \int \limits_{\B^{n}}\left(|\nabla_{\B^{n}}u|^{2}-\lambda u^{2}\right)~ \dvg}{\displaystyle \left(~\int \limits_{\B^{n}}|u|^{2^{\star}}~\dvg\right)^{2/2^{\star}}}=\inf \limits_{u \in C^{\infty}_{c}(\Be^{n})\setminus\{0\}}\frac{\displaystyle \int \limits_{\Be^{n}}\left(|\nabla u|^{2}-\tilde{h}_{\lambda}u^{2}\right)~{\rm d}x}{\displaystyle \left(~\int \limits_{\Be^{n}}|u|^{2^{\star}}~{\rm d}x\right)^{2/2^{\star}}},
\end{align*}
where $\displaystyle \tilde{h}_{\lambda}(x)=\frac{4\lambda-n(n-2)}{\left(1-|x|^{2}\right)^{2}}$. We test the localized version of the Aubin-Talenti bubble $U:=U[0,1]$ denoted by
\begin{align*}
U_{\eps}(x):=\eps^{-\frac{n-2}{2}}U\left(\frac{x}{\epsilon}\right)= (n(n-2))^{\frac{n-2}{4}}\frac{\eps^{\frac{n-2}{2}}}{(\eps^{2} + |x|^2)^{\frac{n-2}{2}}}, \qquad  x \in \R^n, \eps>0.
\end{align*}
Consider a cut-off function $\eta\in C^\infty_c(\Be^{n})$ such that $\eta(x)\equiv1$ in a neighborhood of $0$ and the test functions $u_{\eps}\in D^{1,2}(\B^{n})$ defined by $u_{\eps}(x):=\eta(x)U_{\eps}(x)$ for $x\in \Be^{n}$. We obtain as $\eps \to 0$
\begin{align*}
\int\limits_{\Be^{n}} u_{\eps}^{2^{\star}}~{\rm d}x&=\int \limits_{B(0,\frac{1}{2})} u_{\eps}^{2^{\star}}~{\rm d}x+ \int \limits_{\Be^{n}\setminus B(0,\frac{1}{2})} u_{\eps}^{2^{\star}}~{\rm d}x =\int\limits_{B(0,\frac{1}{2}\eps^{-1})}U^{2^{\star}}~{\rm d}x~+\int \limits _{B(0,\frac{1}{2}\eps^{-1})^c} \eta(\eps x)^{2^{\star}} U^{2^{\star}}~{\rm d}x\\
&=\int\limits_{\R^{n}} U^{2^{\star}}~{\rm d}x+ O \left(\eps^{n}\right).
\end{align*}
Similarly, one also has
\begin{align*}
\int\limits_{\Be^{n}}|\nabla u_{\eps}|^{2}~{\rm d}x &= \int\limits_{B(0,\frac{1}{2})}|\nabla u_{\eps}|^{2}~{\rm d}x +   \int\limits_{\Be^{n}\setminus B(0,\frac{1}{2})}|\nabla u_{\eps}|^{2}~{\rm d}x=\int\limits_{B(0,\frac{1}{2}\eps^{-1})} |\nabla U|^{2}~{\rm d}x+ O\left(\eps^{n-2}\right)\\
&=\int\limits_{\R^{n}} |\nabla U|^{2}~{\rm d}x+ O\left(\eps^{n-2}\right).
\end{align*}
For the lower order term, we obtain as $\eps\to 0$
\begin{align*}
\int\limits_{\Be^{n}} \tilde{h}_{\lambda}u_{\eps}^2~{\rm d}x =
 \left \{ \begin{array} {lc}
           \left(4\lambda-n(n-2)\right)\eps^{2}  \left[~\displaystyle\int \limits_{\R^{n}} U^{2}~{\rm d}x  + o(1)\right] \quad    &\hbox{ for }~ n\geq 5, \\
           \left(4\lambda-n(n-2)\right)\eps^{2} \log \left( \dfrac{1}{\eps}\right) \left[\omega_{n-1} + o(1) \right] \quad  &\hbox{ for  }~n =4.
            \end{array} \right.
\end{align*}
Combining the above estimates, we obtain as $\eps\to 0$,
\[\begin{split}
&S_{n, p,\lambda}(\B^{n})\leq ~{\left( \displaystyle \int \limits_{~\Be^{n}}\left(|\nabla u_{\eps}|^{2}-\tilde{h}_{\lambda}u_{\eps}^{2}\right)~{\rm d}x\right)}{\displaystyle \left(~\int \limits_{\B^{n}}u_{\eps}^{2^{\star}}~{\rm d}x\right)^{-2/2^{\star}}}=\\
&S(\R^{n})- \left(~\displaystyle\int\limits_{\R^{n}} U^{2^{\star}}~{\rm d}x\right)^{-2/2^{\star}}\times
  \begin{cases}
  \left(4\lambda-n(n-2)\right)\eps^{2}\left[~\displaystyle\int\limits_{\R^{n}}U^{2}~{\rm d}x+ o(1)\right] ~\hbox{ for }~ n\geq 5, \\ 
   \left(4\lambda-n(n-2)\right)\eps^{2} \log\left( \dfrac{1}{\eps}\right)\left[\omega_{n-1} + o(1) \right]~\hbox{ for  }~n =4.
        \end{cases}
\end{split}\]
Thus for $\lambda>\frac{n(n-2)}{4}$ it follows that $S_{n,p, \lambda}(\B^{n})<S(\R^{n})$.

\end{proof}

\section{Stability \'{a} la Bianchi-Egnell: Proof of Theorem~\ref{maintheorem-1}}\label{proof1}

In this section, we give the proof of Theorem~\ref{maintheorem-1}. The next lemma is crucial for the proof and follows the lines of \cite{BE}.

\subsection{Behaviour near $\mathcal{Z}_0$}
\begin{lemma}\label{almostlemma}
Let $n \geq 3$ and $\lambda,p$ satisfies the hypothesis {\bf (H1)}, then there exists a constant $\alpha > 0$ depending on the dimension $n,\lambda$ and $p$ such that there holds 
\begin{equation*}
 \int \limits_{\B^{n}}\left(|\nabla_{\B^{n}}\phi|^{2}  -\lambda \phi^{2}\right)~\dvg \, - \, S_{n, p,\lambda}(\B^{n})\left(~\int \limits_{\B^{n}}|\phi|^{p+1}~\dvg
 \right)^{\frac{2}{p+1}} \geq \alpha \, \hbox{dist}\left(\phi, \mathcal{Z}_{0}\right)^{2} \, + \, \circ(\hbox{dist}\left(\phi, \mathcal{Z}_{0}\right)^{2}),
\end{equation*}
for all $\phi$ satisfying $\hbox{dist}\left(\phi, \mathcal{Z}_{0}\right) < \|\phi\|_{\lambda}.$

\end{lemma}

\begin{proof}
We first recall that $\mathcal{Z}_{0}$ is the $(n+1)$ dimensional manifold consisting of $\{ c\, \calu_b : c \in \mathbb{R}\backslash \{0\}, b \in \bn\}$, where $\calu_b:=\calu\circ\tau_b$. We further recall that $\|\calu_b\|^2_{\lambda} = S_{n,p,\lambda}(\bn)^{\frac{p+1}{p-1}}$. For any $\phi\in H^1(\bn)$
\begin{align}
\hbox{dist}\left(\phi, \mathcal{Z}_{0}\right)^{2} \,& = \, \inf_{c \in \mathbb{R}, \,  b \in \bn} \|  \phi \, - \, c\, \calu_b  \|_{\lambda}^2 \notag \\
& = \inf_{c \in \mathbb{R}, \, b \in \bn} \left(\int_{\bn} |\nabla_{\bn} (\phi - c \, \calu_b)|^2 \, \dvg \, - \, \lambda \int_{\bn} (\phi \, -\, c \, \calu_b)^2 \, \dvg \right)\notag \\
&  = \inf_{c \in \mathbb{R}, \, b \in \bn} \left(\int_{\bn}  \left( |\nabla_{\bn} \phi|^2 \, -\, \lambda \phi^2 \right) \, \dvg  \,+ \, c^2  \,  
\int_{\bn} \left( |\nabla_{\bn} \calu_b|^2 \, - \, \lambda \, \calu_b^2 \right) \dvg \notag \right.\\
& \quad\quad\left.- 2 \, c \int_{\bn} \left( \langle\nabla_{\bn} \phi , \nabla_{\bn} \calu_b\rangle_{\bn} \, - \, \lambda \, \phi \, \calu_b \right) \, \dvg \right)\notag \\
&=\inf_{c \in \mathbb{R}, \,  b \in \bn} \left[ \|\phi\|^2_{\lambda} \, + \, c^2 S_{n,p,\lambda}(\bn)^{\frac{p+1}{p-1}}
 - 2 \, c \int_{\bn} \left( \langle\nabla_{\bn} \phi , \nabla_{\bn} \calu_b\rangle_{\bn} \, - \, \lambda \, \phi \, \calu_b \right) \, \dvg  \right]
\end{align}
According to our assumption $\phi\in H^1(\bn), $ satisfies
\begin{align}\label{4-6-3}
\hbox{dist}\left(\phi, \mathcal{Z}_{0}\right)^{2} \, & =\, \inf_{c \in \mathbb{R}, \,  b \in \bn} \left[ \|\phi\|^2_{\lambda} \, + \, c^2 S_{n,p,\lambda}(\bn)^{\frac{p+1}{p-1}}
 - 2 \, c \int_{\bn} \left(\langle\nabla_{\bn} \phi , \nabla_{\bn} \calu_b\rangle_{\bn} \, - \, \lambda \, \phi \, \calu_b \right) \, \dvg  \right] \notag \\
 & < \|\phi\|^2_{\lambda}.
\end{align}
 As a result, the infimum is achieves by some $(c_0, b_0) \in \mathbb{R} \times \bn.$ Indeed, for a minimizing sequence $(c_k,b_k),$ $c_k^2$ being the dominating term, $c_k$ remains bounded. If 
$d(b_k,0) \rightarrow \infty,$ then $\calu_{b_k} : = \calu \circ \tau_{b_k} \rightharpoonup 0$ in $H^1(\bn),$ and therefore the last term in \eqref{4-6-3} converges to zero, contradicting the assumption $\hbox{dist}\left(\phi, \mathcal{Z}_{0}\right) < \|\phi \|_{\lambda}.$  For the same reason
$c_0 \neq 0.$  

We denote $\phi_i:=\frac{d}{dt}|_{t=0}U\circ \tau_{b_0+te_i}$, then  $\phi_i$ are the eigen functions $(-\Delta_{\bn} - \lambda)/\mathcal{U}_{b}^{p-1},$ (see Remark~\ref{r:4-6-1}).
\medskip

Next we claim that $\phi \, - \, c_0 \, \calu_{b_0} \, \perp \calu_{b_0}, \phi_1, \ldots, \phi_n$ in 
$L^2_{\calu_{b_0}^{p-1}} (\bn)$. To this end, we first take variation with respect to $c,$ and we obtain 
\begin{align*}
c_0 S_{n,p,\lambda}(\bn)^{\frac{p+1}{p-1}}\,&  = \, \int_{\bn} \left( \langle\nabla_{\bn} \phi , \nabla_{\bn} \, \calu_{b_0} \rangle_{\bn}\, - \, \lambda \, \phi \, \calu_{b_0} \right) \, \dvg
= \int_{\bn} \phi \, \left( - \De_{\bn} \calu_{b_0} \, - \, \lambda \, \calu_{b_0} \right) \, \dvg \notag \\
& =  \int_{\bn} \phi \, \calu^{p}_{b_0} \, \dvg.
\end{align*}
Therefore from the above computation and the fact that $ \int_{\bn}  \calu^{p+1}_{b_0} \, \dvg = S_{n,p,\lambda}(\bn)^{\frac{p+1}{p-1}},$ we conclude 
$$
\int_{\bn} \left( \phi - c_0 \, \calu_{b_0}\right) \, \calu_{b_0} \,  \calu_{b_0}^{p-1} \, \dvg \, = \, 0.
$$
This proves that $\phi \, - \, c_0 \, \calu_{b_0} \, \perp \calu_{b_0}$ in $L^2_{\calu_{b_0}^{p-1}} (\bn).$ Now taking the variation with respect to $b,$ we get 
for $i =1, \ldots, n,$
\begin{align*}
0 \, & = \, \int_{\bn} \left( \langle\nabla_{\bn} \phi , \nabla \phi_i \rangle_{\bn}\, - \, \lambda \, \phi \, \phi_i \right) \, \dvg 
= \int_{\bn} \phi \left( - \De \phi_i \, - \, \lambda \, \phi_i \right) \, \dvg \notag \\
& = p  \int_{\bn} \phi \,  \, \calu_{b_0}^{p-1} \, \phi_i \, \dvg.
\end{align*}
From above and further using the fact that $\int_{\bn} \calu^{p-1}_{b_0} \, \phi_i \, \dvg = 0$ (which is a consequence of the conformal invariance of the norm) we deduce 
\begin{align*}
\int_{\bn} \left( \phi - c_0 \, \calu_{b_0} \right) \, \phi_i \, \calu^{p-1}_{b_0} \, \dvg \, = \, 0.
\end{align*}
Therefore 
\begin{equation}\label{4-6-4}
\phi - c_0 \, \calu_{b_0} \, \perp \, (T \mathcal{Z}_{0})_{(c_0, b_0)},
\end{equation}
where $(T \mathcal{Z}_{0})_{(c_0, b_0)}$ is tangent space at the point $c_0\calu_{b_0}.$ Now using Proposition~\ref{eigen value lemma} and min-max characterisation 
of eigenvalues we derive 

$$
\mu_3 \leq \dfrac{\| \phi - c_0 \, \calu_{b_0}\|^2_{\lambda}}{\int_{\bn} (\phi - c_0 \, \calu_{b_0})^2 \, \calu_{b_0}^{p-1} \, \dvg}.
$$

Now let us write $\phi - c_0 \, \calu_{b_0} = \mathcal{D} v,$ for some $v \in H^1(\bn)$ with $\| v \|_{\lambda} = 1,$ $v \perp (T \mathcal{Z}_{0})_{(c_0, b_0)},$
and $\mathcal{D} := \, \hbox{dist}\left(\phi, \mathcal{Z}_{0}\right).$ We aim to compute the $\| \phi \|_{L^{p+1}(\bn)}.$ To this end we expand $| \phi |^{p+1}$ using Taylor's series 
\begin{align}
|\phi|^{p+1} \, & = \, | c_0 \, \calu_{b_0} + \cald\, v|^{p+1} = |c_0|^{p+1} \, \calu_{b_0}^{p+1} \, + \, {(p+1)} \, |c_0|^{p-1} c_0 \, \calu_{b_0}^{p} \, \cald v  \notag \\
& \qquad+ \frac{p(p+1)}{2} \, |c_0|^{p-1} \, \calu_{b_0}^{p-1} \, (\cald v)^2 \, + \, \circ(\cald^2) \notag.
\end{align}
Integrating the above identity and using the fact that $v \perp \calu_{b_0}$ in $L^2_{\calu_{b_0}^{p-1}}(\bn),$ we obtain 
\begin{align}\label{2star}
\left( \int_{\bn} |\phi|^{p+1} \, \dvg \right)^{\frac{2}{p+1}} \,&  = \, \left( |c_0|^{p+1} S_{n, p,\lambda}(\bn)^{{\frac{p+1}{p-1}}} \,
+ \, \frac{p(p+1)}{2} \cald^2 \, |c_0|^{p-1} \int_{\bn} \calu^{p-1} \, v^2 \, \dvg \, + \, \circ(\cald^2) \right)^{\frac{2}{p+1}} \notag \\
& \leq \left( |c_0|^{p+1} S_{n, p,\lambda}(\bn)^{{\frac{p+1}{p-1}}} \, + \,  \frac{p(p+1)}{2} \cald^2 \, |c_0|^{p-1} \frac{1}{\mu_3} \, 
+ \, \circ(\cald^2)\right)^{\frac{2}{p+1}} \notag \\
& \leq \left( |c_0|^{p+1} S_{n,p, \lambda}(\bn)^{{\frac{p+1}{p-1}}}\right)^{\frac{2}{p+1}} \notag \\
 & \quad + 
 \frac{2}{p+1} \left(|c_0|^{p+1} S_{n, p,\lambda}(\bn)^{{\frac{p+1}{p-1}}} \right)^{\frac{2}{p+1} -1} \notag 
 \times \left( \frac{p(p+1)}{2} \cald^2 \, |c_0|^{p-1} \frac{1}{\mu_3}  + \circ(\cald^2) \right) \notag \\
& = |c_0|^2 \, (S_{n, p,\lambda}(\bn))^{\frac{p+1}{p-1}-1} \, + \, \frac{p}{\mu_3} \, (S_{n, p,\lambda}(\bn))^{-1} \, \cald^2 \, + \, \circ(\cald^2)\notag\\
&= |c_0|^2 \, (S_{n, p,\lambda}(\bn))^{\frac{p+1}{p-1}-1} \, + \, S_{n, p,\lambda}(\bn)^{-1}\frac{\mu_2}{\mu_3} \cald^2 \, + \, \circ(\cald^2), 
\end{align}
where in the third inequality we have used $(1 + t)^{\alpha} \leq 1 + \alpha t,$  $t >0,\, \alpha \in (0,1)$ and in the last inequality we have used $\mu_2 = p$.

%

Using the orthogonality condition we see that $\cald^2 \, = \, \|\phi\|^2_{\lambda} \, - \, |c_0|^2(S_{n, p,\lambda}(\bn))^{\frac{p+1}{p-1}}$.
Indeed, from \eqref{4-6-3}, we have
$$\cald^2 = \|\phi\|^2_{\lambda} \, + \, c_0^2 (S_{n,p, \lambda}(\bn))^{\frac{p+1}{p-1}}
 - 2 \, c_0 \int_{\bn} \left( \langle\nabla_{\bn} \phi , \nabla_{\bn} \calu_{b_0} \rangle_{\bn} \, - \, \lambda \, \phi \, \calu_{b_0} \right) \, \dvg.$$
 Now, 
\begin{align*}
& \ \ \ \ \ \ - 2 \, c_0 \int_{\bn} \left( \langle\nabla_{\bn} \phi , \nabla_{\bn} \calu_{b_0} \rangle_{\bn} \, - \, \lambda \, \phi \, \calu_{b_0} \right) \, \dvg\\
&\quad=-2\int_{\bn} \Big( \langle \nabla_{\bn}(c_0 \, \calu_{b_0} + \mathcal{D} v), \nabla_{\bn}( c_0\calu_{b_0}) \rangle_{\bn}\,-\,\lambda(c_0 \, \calu_{b_0} + \mathcal{D} v)c_0\calu_{b_0} \Big)\,\dvg\\
&\quad=-2\|c_0 \, \calu_{b_0} \|_\lambda^2=-2c_0^2 S_{n, p,\lambda}(\bn))^{\frac{p+1}{p-1}},
\end{align*}
where in the second equality we used the fact that $v \perp c_0 \, \calu_{b_0}$ in $L^2_{\calu_{b_0}^{p-1}}(\bn)$.  

Plugging these in \eqref{2star} yields
$$S_{n, p,\lambda} \| \phi \|^2_{L^{p+1}}\leq \|\phi \|^2_{\lambda} \, - \, \cald^2 \, + \, \frac{\mu_2}{\mu_3} \cald^2 \, + \, \circ(\cald^2).$$
Therefore,
$$
\| \phi \|^2_{\lambda} \, - \, S_{n, p,\lambda}(\bn) \| \phi \|_{L^{p+1}}^2 \geq \cald^2 \left(1 - \frac{\mu_2}{\mu_3} \right) \, + \, \circ(\cald^2),
$$
whenever, $\phi$ satisfying $\hbox{dist}\left(\phi, \mathcal{Z}_{0}\right) < ||\phi||_{\lambda}.$
This completes the proof. 

\end{proof}

The next non-quantitative stability result follows from the Palais-Smale decomposition of 
\cite{BS}.

\begin{lemma}\label{minimisationlemma}
Let $n \geq 3,$ and $\lambda, p$ satisfies the hypothesis {\bf (H1)}. Let $\varphi_k \in H^1(\bn)$ be a sequence  such that 

$$
\| \varphi_k \|_{\lambda}^2 \, = \, S_{n,p,\lambda}(\bn)^{\frac{p+1}{p-1}}, \quad \hbox{and} \quad S_{n, p,\lambda}(\bn) \, \| \varphi_k \|^2_{L^{p+1}(\bn)} \rightarrow  S_{n,p,\lambda}(\bn)^{\frac{p+1}{p-1}}.
$$
Then (upto a subsequence) there exists a sequence $\{ b_k \} \subset \bn$ such that 

\begin{equation}\label{convergence}
\| \varphi_k \, - \, \calu \circ \tau_{b_k} \|_{\lambda} \rightarrow 0. 
\end{equation}

\end{lemma}

\begin{proof}
It is easy to see from the hypothesis on $\varphi_{k}$ that it is a minimizing sequence for $S_{n,p,\lambda}(\bn).$ Using Ekeland's variational principle, we get a Palais-Smale sequence for the functional $I_{\lambda}$ (defined in \eqref{ilambda}) at a level $(\frac{1}{2} - \frac{1}{p+1})S_{n, p,\lambda}(\bn)^{\frac{p+1}{p-1}}$. We denote the extracted sequence again by $\varphi_{k}$ itself and we may also assume that $\varphi_{k}$ satisfies $\| \varphi_k \|_{\lambda}^2 \, = \, S_{n,p,\lambda}(\bn)^{\frac{p+1}{p-1}}$. In the subcritical case we can directly apply   \cite[Theorem~3.3]{BS} and conclude \eqref{convergence} holds. For the critical case 
combining Proposition~\ref{PEUCS}
and  Palais-Smale decomposition \cite[Theorem~3.3]{BS} we conclude that there can not exist any Aubin-Talenti bubbles in  Palais-Smale decomposition and hence there exists a sequence  $b_k \subset \bn$ such that \eqref{convergence} holds. 

\end{proof}

\noindent
{\bf Proof of Theorem~\ref{maintheorem-1}:}

\begin{proof}

Now we shall begin the proof of Theorem~\ref{maintheorem-1}. Assume, on the contrary, there exists a sequence $\{\varphi_k\} \in H^1(\bn)$ such that

\begin{align}\label{impossibility}
\dfrac{\int_{\bn} \left( |\nabla_{\bn} \varphi_k|^2 \, - \, \lambda \varphi_k^2 \right) \, \dvg \, -
 \, S_{n, p,\lambda}(\bn) \| \varphi_k \|^2_{L^{p+1}(\bn)}}{\hbox{dist}\left(\phi_k, \mathcal{Z}_{0}\right)^2} \rightarrow 0, \quad \hbox{as} \ k \rightarrow \infty.
\end{align}
By scaling property of $dist$, we may assume that $\| \varphi_k\|^2_{\lambda} = S_{n,\lambda}(\bn)^{\frac{p+1}{p-1}},$ for all $k$ and 
let us assume $L = \limsup_{k \rightarrow \infty} \hbox{dist}\left(\varphi_k, \mathcal{Z}_{0}\right).$ Moreover, we have 

\begin{align*}
\hbox{dist}\left(\varphi_k, \mathcal{Z}_{0}\right)^2 \, = \, \| \varphi_k \|^2_{\lambda} \, - \, c_k^2 \leq \| \varphi_k \|_{\lambda} \, = \, S_{n,p,\lambda}(\bn)^{\frac{p+1}{p-1}},
\end{align*}
where $c_k$ is the optimizer in $\hbox{dist}\left(\varphi_k, \mathcal{Z}_{0}\right) \, = \, \inf_{c \in \mathbb{R}, b \in \bn} \| \varphi_k \, - \, c\, \calu_b \|_{\lambda}$ and hence 
$L \in [0, 1].$

\medskip 

Now we have two possibilities either $L = 0$ or $L \neq 0.$ One can immediately see  $L =0$ would contradict  Lemma~\ref{almostlemma}. We are left with the case 
$L > 0.$ In this case we have 

$$
\underbrace{\int_{\bn} \left( |\nabla_{\bn} \varphi_k|^2 \, - \, \lambda \varphi_k^2 \right) \, \dvg}_{:= \| \varphi_k\|^2_{\lambda}} \, - 
\, S_{n, p,\lambda} \, \| \varphi_k \|^2_{L^{p+1}(\bn)} \rightarrow 0
$$
and using $\| \varphi_k\|^2_{\lambda} \, = S_{n,p,\lambda}(\bn)^{\frac{p+1}{p-1}}\, ,$ we obtain $S_{n, p,\lambda} \| \varphi_k\|^2_{L^{p+1}(\bn)} \rightarrow S_{n,p,\lambda}(\bn)^{\frac{p+1}{p-1}}.$ Therefore by Lemma~\ref{minimisationlemma}, there 
exists a sequence $\{b_k \} \subset \bn$ such that 

$$
\| \varphi_k \, - \, \calu \circ \tau_{b_k}  \|_{\lambda} \rightarrow 0,
$$
which in turn implies $\hbox{dist}\left(\varphi_k, \mathcal{Z}_{0}\right) \rightarrow 0,$ contradicting our assumption on $L.$ Hence \eqref{impossibility} is impossible and we conclude the proof.
 
\end{proof}

\section{Stability from Euler-Lagrange point of view: Proof of Theorem~\ref{maintheorem-2}}\label{proof2}
This section is devoted to the proof of Theorem~\ref{maintheorem-2} and Corollary \ref{corthm2}.

\medskip

\noindent
{\bf Proof of Theorem~\ref{maintheorem-2}}: 
\begin{proof}
According to our assumption $\| I_{\lambda}^{\prime}(u_{\alpha}) \|_{H^{-1}(\bn)}$ is small. 
From Theorem~\ref{t:BS1}, $u_{\alpha}$ is close to the manifold $\mathcal{Z}$ i.e. there exists $b_{\alpha}\in\bn$ such that
\begin{align*}
\|u_{\alpha}-\calu_{b_{\alpha}}\|_{\lambda} = o(1), \  \ \mbox{as} \ \alpha \rightarrow \infty.
\end{align*}
In particular, $\|u_{\alpha} \|_{\lambda}^2 = \|\calu_{b_{\alpha}}\|_{\lambda}^2 + o(1) = S_{n,p,\lambda}(\bn)^{\frac{p+1}{p-1}} + o(1)$ as $\alpha \rightarrow \infty.$
In the remaining analysis we drop the index $\alpha$ and we simply denote $u_{\alpha}$ by $u, b_{\alpha}$ by $b$ and so on.
As before, we consider the minimization problem $\inf_{c \in \mathbb{R}, \,  b \in \bn} \|  u \, - \, c\, \calu_b  \|_{\lambda}^2$ and claim that the minimum is achieved for some $b_0\in\bn$ and $c_0\neq 0$.

Indeed, for a minimizing sequence $\{(c_l,b_l\})\}$, $c_l$ remains bounded ($c_l^2$ being the leading term). If $d(b_l,0)\to \infty$ then $\calu_{b_l}\rightharpoonup 0$ in $H^1(\bn)$. This implies 
\begin{align*}
\|  u \, - \, c_l\, \calu_{b_l}  \|_{\lambda}^2=\|u\|^2_{\lambda}+c_l^2+o(1).
\end{align*}
This in turn implies $\|  u \, - \, \, \calu_{b}  \|_{\lambda}$ can not be made arbitrarily small. Hence the claim follows. 
We set $u-c_0\calu_{b_0}=v$. Then as before,  
$v \, \perp \, (T \mathcal{Z}_{0})_{(c_0, b_0)}$ (see \eqref{4-6-4}), where 
$$(T \mathcal{Z}_{0})_{(c_0, b_0)}=\{\calu_{b_0}, v_1,\cdots, v_n\}$$
and $v_i=\frac{d}{dt}|_{t=0}\calu_{b_0+te_i}$. Therefore,
\begin{equation}\label{4-6-5}
 \int_{\bn} \left(\langle \nabla_{\bn} v , \nabla_{\bn} \calu_{b_0} \rangle_{\bn} \, - \, \lambda \, v \, \calu_{b_0} \right) \, \dvg=0
\end{equation}
\begin{equation*}
 \int_{\bn} \left(\langle \nabla_{\bn} v , \nabla_{\bn} v_i \rangle_{\bn} \, - \, \lambda \, v \, v_i \right) \, \dvg=0
\end{equation*}
and $\|v\|_{\lambda} = o(1)$. Moreover, note that by the above orthogonality conditions, $u \geq 0$ and $\|  u \, - \, c_0\, \calu_{b_0}  \|_{\lambda}^2 = o(1)$ implies
\begin{align*}
S_{n,p,\lambda}(\bn)^{\frac{p+1}{p-1}} + o(1) = \|u_{\alpha} \|_{\lambda}^2 = |c_0|^2 \| \calu_{b_0} \|_{\lambda}^2 + \| v \|_{\lambda}^2 = |c_0|^2S_{n,p,\lambda}(\bn)^{\frac{p+1}{p-1}} + o(1)
\end{align*}
and hence $|c_0 - 1| = o(1).$

Since $\calu_{b_0}$ solves \eqref{eq1} and $v_i$ are the second eigenvectors of the operator $(-\Delta_{\bn} - \lambda)/\calu_{b_0}^{\tstar - 2}$ we get
\begin{equation}\label{4-6-7} 
\int_{\bn}\calu_{b_0}^{p}v \ \dvg= \int_{\bn}v v_i\calu_{b_0}^{p-1} \ \dvg =0,
\end{equation}
and hence $v$ is orthogonal to 1st and 2nd eigenspace. Therefore
$$\|v\|_{\lambda}^2 \geq \mu_3\int_{\bn}  v^2 \calu_{b_0}^{p-1} \ \dvg,$$
 and by Proposition~\ref{eigen value lemma}, $\mu_3 >p$.
Therefore by \eqref{4-6-5}
\begin{align}\label{4-6-8}
\|v\|_{\lambda}^2 &=\int_{\bn}[\langle \nabla_{\bn}v, \nabla_{\bn}(u-c_0\calu_{b_0}) \rangle_{\bn}-\lambda v(u-c_0\calu_{b_0})]\,\dvg\notag\\
&=\int_{\bn}(\langle \nabla_{\bn}v, \nabla_{\bn}u \rangle_{\bn} -\lambda vu)\,\dvg\notag\\
&=\int_{\bn}(-\Delta_{\bn}u-\lambda u)v\,\dvg\notag\\
&=\int_{\bn}(-\Delta_{\bn}u-\lambda u-u^{p})v\,\dvg+\int_{\bn}u^{p}v\,\dvg\notag\\
&\leq\|-\Delta_{\bn}u-\lambda u-u^{p}\|_{H^{-1}(\bn)}\|v\|_{H^1(\bn)}+\int_{\bn}(c_0\calu_{b_0}+v)^{p}v\, \dvg\notag\\
&\leq C(n,\lambda)\|-\Delta_{\bn}u-\lambda u-u^{p}\|_{H^{-1}(\bn)}\|v\|_{\la}+\int_{\bn}(c_0\calu_{b_0}+v)^{p}v\, \dvg.
\end{align}
Now using \eqref{4-6-7}, we estimate the last integral of the RHS of \eqref{4-6-8} as follows 
$$
\int_{\bn}(c_0\calu_{b_0}+v)^{p}v\, \dvg=p\int_{\bn}c_0^{p-1}\calu_{b_0}^{p-1}v^2+O \left(\sup_{q, r} \int_{\bn}|v|^{q+1}\calu_{b_0}^{r} \dvg\right)
$$
where the sup is taken over all  $q, r $ satisfying $2\leq q\leq p$ and $q+r=p$.  Applying H\"older and Poincar\'e-Sobolev inequality to $ \int_{\bn}|v|^{q+1}\calu_{b_0}^{r} \dvg$ for the two extreme cases $q = 2$ and $p$ we see that the terms can be bounded 
by $\|v\|_{\lambda}^{\frac{3}{2}}$ and $\|v\|_{\lambda}^{\frac{p+1}{2}}$ respectively.
Therefore, by interpolation we get
\begin{equation}\label{4-6-9}
\int_{\bn}(c_0\calu_{b_0}+v)^{p}v\, \dvg\leq \frac{p}{\mu_3}c_0^{p-1}\|v\|_{\lambda}^2+O(\|v\|_{\lambda}^{1+\gamma}).
\end{equation}
where $\gamma=\min\{\frac{1}{2}, \frac{p-1}{2}\}$. Substituting \eqref{4-6-9} into \eqref{4-6-8} yields
$$\left(1- \frac{p}{\mu_3}c_0^{p-1}\right)\|v\|_{\lambda}^2\leq C(n,\lambda)\|-\Delta_{\bn}u-\lambda u-u^{p}\|_{H^{-1}(\bn)}\|v\|_{\la}+O(\|v\|_{\lambda}^{1+\gamma}),$$
Since $|c_0-1| = o(1)$ and $\mu_3>p$, the above inequality yields
$$\|v\|_{\lambda}^2\leq C(n, \lambda) \|I'_{\lambda}(u_{\alpha)}\|_{H^{-1}(\B^{n})}\|v\|_{\lambda}+O(\|v\|_{\lambda}^{1+\gamma}).$$
 $\|v\|_{\lambda}$ being small, it is easy to obtain from the above inequality
$$\|v\|_{\lambda}\leq C(n, \lambda) \|I'_{\lambda}(u_{\alpha)}\|_{H^{-1}(\B^{n})}.$$

Next we need to control $|c_0 - 1|$ quantitatively in terms of $\| I_{\lambda}^{\prime}(u) \|_{H^{-1}(\bn)}.$ Note that 
\begin{align} \label{subz1}
\|u \|_{\lambda}^2 = c_0^2 \|\calu_{b_0}\|_{\lambda}^2 + \|v \|_{\lambda}^2 = c_0^2S_{n,p,\lambda}(\bn)^{\frac{p+1}{p-1}} + O(\| I_{\lambda}^{\prime}(u) \|_{H^{-1}(\bn)}^2)
\end{align}

and using $\int_{\bn} \calu_{b_0}^{p}v = 0$ 

\begin{align}\label{subz2}
\int_{\bn} u^{p+1} &= c_0^{p+1} \int_{\bn} \calu_{b_0}^{p+1} + (p+1) |c_0|^{p-1}c_0\int_{\bn} \calu_{b_0}^{p}v +  O(\| I_{\lambda}^{\prime}(u) \|_{H^{-1}(\bn)}^2), \notag\\
&= c_0^{p+1}S_{n,p,\lambda}(\bn)^{\frac{p+1}{p-1}}  + O(\| I_{\lambda}^{\prime}(u) \|_{H^{-1}(\bn)}^2).
\end{align}

Subtracting \eqref{subz2} from \eqref{subz1} we get 
\begin{align*}
(I_{\lambda}^{\prime}(u), u) = c_0^2S_{n,p,\lambda}(\bn)^{\frac{p+1}{p-1}}  - c_0^{p+1}S_{n,p,\lambda}(\bn)^{\frac{p+1}{p-1}}  +  O(\| I_{\lambda}^{\prime}(u) \|_{H^{-1}(\bn)}^2).
\end{align*}

As a result we get $ c_0^2S_{n,p,\lambda}(\bn)^{\frac{p+1}{p-1}}  - c_0^{p+1}S_{n,p,\lambda}(\bn)^{\frac{p+1}{p-1}} = O(\| I_{\lambda}^{\prime}(u) \|_{H^{-1}(\bn)})$ and hence $|c_0 - 1| =  O(\| I_{\lambda}^{\prime}(u) \|_{H^{-1}(\bn)}).$ Hence 
\begin{align*}
\| u - \calu_{b_0}\|_{\lambda} \leq \| u - c_0\calu_{b_0}\|_{\lambda} + |c_0 - 1|\|\calu_{b_0}\|_{\lambda} \leq C(n,\lambda)\| I_{\lambda}^{\prime}(u) \|_{H^{-1}(\bn)}.
\end{align*}
as desired.
\end{proof}

\medskip 

\noindent
{\bf Proof of Corollary \ref{corthm2}.}
\begin{proof} According to our assumption and Poincar\'{e}-Sobolev inequality we conclude that if $\epsilon_0$ is small then 
\begin{align} \label{some inequality}
\left(\frac{1}{2} - \frac{1}{p+1} - 2\epsilon_0\right)S_{n,p,\lambda}(\bn)^{\frac{p+1}{p-1}} \leq I_{\lambda}(u) \leq \left(\frac{1}{2} - \frac{1}{p+1} +\epsilon_0\right) S_{n,p,\lambda}(\bn)^{\frac{p+1}{p-1}}.
\end{align}

If $\| I_{\lambda}^{\prime}(u) \|_{H^{-1}(\bn)} \geq \delta_0>0$ for some preassigned $\delta_0$ then clearly the statement of the corollary holds with $C(n,p,\lambda) = \frac{2}{\delta_0}S_{n,p,\lambda}(\bn)^{\frac{p+1}{2(p-1)}}.$ On the other hand, if for a sequence $\{u_{\alpha}\},\| I_{\lambda}^{\prime}(u_{\alpha}) \|_{H^{-1}(\bn)} \rightarrow 0,$ then by \eqref{some inequality} and Theorem \ref{t:BS1} we conclude that the hypothesis of the Theorem \ref{maintheorem-2} holds and hence the result follows.
\end{proof}

\section{Stabilization and extinction profile of the fast diffusion flow}\label{fd sec}

In this section, we consider the fast diffusion equation on the hyperbolic space.
\begin{equation}\label{FDE}
\begin{cases}
\partial_t u = \Delta_{\bn} u^m, \ \ \ \mbox{on} \ \bn \times (0,T)\\
u = u_0 \geq 0, \ \ \ \ \mbox{on} \ \bn \times \{0\},
\end{cases}
\end{equation}
where $n \geq 3,$ the range of $m < 1$ to be specified in a moment.
The existence, uniqueness and smoothing effect of the finite energy solutions have been studied for a fairly general class of initial data \cite{BGV} on the general Cartan-Hadamard manifold which includes the hyperbolic space. It has been observed that the smoothing effect of the solutions relies on the manifolds which support a Sobolev inequality, as the hyperbolic space does (corresponding to $\lambda = 0$) \cite[Theorem 4.1]{BGV}. On the other hand, in addition to the Sobolev inequality, if the manifold supports a Poincar\'e inequality, as the hyperbolic space does, then \eqref{FDE} gives rise to an interesting finite time extinction property of the solutions for every $m \in (0,1)$ \cite[Theorem 6.1]{BGV}. 
\medskip
Several mathematicians have extensively studied the fast diffusion equations on domains of $\rn$ and manifolds. It is beyond the scope of this article to give even a minimal account description of this subject. We refer the interested readers to the articles \cite{DiFD1, DiFD2, FSFD, DDFD, FD-1, FD-2, FD-3, FD-4, FD0, FD1, FD2, FD3, FD4, FD5,  FD6, FD7} and the references cited therein for a broad account of this subject. Moreover, several improvements in the family of Gagliardo-Nirenberg-Sobolev, and Hardy-Littlewood-Sobolev inequalities on the Euclidean space based on the fast diffusion equation have been obtained in \cite{DT, DO}.

\medskip

In this section, we discuss the sharp asymptotic of the solutions near the extinction time for radial initial data, which was first studied on the hyperbolic space in \cite{GMo}. It has been observed that near the extinction time, the asymptotic of the solutions are given by the separation of the variable solutions whose spatial component satisfies the equation $-\Delta_{\bn} V = V^{\frac{1}{m}}.$ In principle, there may exist infinitely many positive radial solutions $V$  but only one of them has finite energy, which is denoted by $\calu$ and defined in \eqref{calu} \cite{BGGV, MS}. For radial data, the profile of the solutions near the extinction time has been obtained by comparison argument and suitably constructing barriers. To state the known results let us formulate our assumptions of this section in the following hypothesis: 

\begin{align*}
\mbox{{\bf (H2)}}
\begin{cases}
 m \in (m_s,1) \ \mbox{where} \ m_s = \frac{1}{\tstar - 1} = \frac{n-2}{n+2}. \ \mbox{Set} \ p = \frac{1}{m},  \ \mbox{then} \ 1 < p < \tstar -1. \\ \\
\mbox{ The initial data} \ u_0 \ \mbox{is assumed to be non-negative (not identically zero) and satisfying} \\
u_0 \in L^q(\bn) \ \mbox{for some} \ q > \max\left\{1, \frac{n(1-m)}{2}\right\}, \ \mbox{and} \ u_0^m \in H^1_{\mbox{\tiny{rad}}}(\bn).
\end{cases}
\end{align*}

Under the above assumptions on the initial data and the exponent, it is shown in \cite{BGV} that on a general Cartan-Haramard manifold which includes the hyperbolic space,
\begin{itemize}
\item[(i)]  there exists a unique solution $u$ to \eqref{FDE} satisfying $u^m \in L^2_{\mbox{\tiny{loc}}}(0,T; H^1(\bn))$ which is radial for any fixed time. Moreover, for a small time, $u$ remains positive.
\item[(ii)] Since the $L^2$-bottom of the spectrum of $-\Delta_{\bn}$ is positive there exists a smallest time $T := T(u_0) < \infty$ such that $u \equiv 0$ for all $t \geq T.$ The $T$ is called the {\it extinction time} and has been emphasized in the equation \eqref{FDE}. \\

Moreover, the exact asymptotic near the extinction time has been obtained in \cite{GMo}.
\item[(iii)] The profile of a given solution $u$ to \eqref{FDE} starting from a radial initial data $u_0$ satisfying the hypothesis {\bf (H2)} is given by the separable solution of the form 
\begin{align}\label{calum}
\calu_{m}(x,t) : = (1-m)^{\frac{1}{1-m}}(T-t)^{\frac{1}{1-m}}\calu^{\frac{1}{m}}(x)
\end{align}
 where $\calu$ is the unique, radial, energy solution to 
 \begin{align*}
 -\Delta_{\bn} \calu = \calu^{\frac{1}{m}}, \ \ \ \calu \in H^1(\bn).
 \end{align*}

\end{itemize}

In \cite[Theorem 1.1]{GMo} the authors deduce uniform convergence in the relative error, that is
\begin{align} \label{conv_1fd}
\lim_{t \rightarrow T-} \left \| \frac{u(\cdot, t)}{\calu_m(\cdot,t)} - 1\right\|_{L^{\infty}(\bn)} = 0.
\end{align}
 
However, no quantitative rate of convergence was mentioned, which we would like to partially address in this and the forthcoming sections. To state the main result of this section, we accumulate a few prerequisites.



\medskip

\noindent
{\bf Rescaled variable.} Following \cite{CFM, FG, GMo} we transform the equation in the logarithmic time scale
\begin{align*}
w(x,\tau) := \left(\frac{u(x,t)}{(1-m)^{\frac{1}{1-m}}(T-t)^{\frac{1}{1-m}}}\right)^m
\end{align*}
where $\tau = \frac{1}{1-m}T\ln \left(\frac{T}{T-t}\right).$ Then direct computation shows that $w$ varifies 
\begin{align} \label{rFDE}
\begin{cases}
\partial_{\tau} w^p = \Delta_{\bn} w + w^p, \ \mbox{on} \ \bn \times (0,\infty), \\
w = c(m,T)u_0^m \in H^1_{\mbox{\tiny{rad}}}(\bn) \ \mbox{on} \ \bn\times \{0\}
\end{cases}
\end{align}
where $ p = \frac{1}{m} \in (1,\tstar -1), c(m,T) = (1-m)^{\frac{m}{1-m}}T^{\frac{m}{1-m}}$.

\medskip

\noindent
{\bf Decay estimates.} The unique positive radial solution $\calu$ satisfies the following sharp asymptotic: there exists a constant $A_0 >0$ such that 
\begin{align*}
A_0^{-1} e^{-(n-1)r} \leq \calu(x) \leq A_0 e^{-(n-1)r} \ \ \mbox{for all} \ x \in \bn,
\end{align*}
where $r = dist(0,x)$ (\cite[Lemma 6.1]{BGGV}, also see \cite[Lemma 3.3]{MS}). In addition, the radial derivative $\calu^{\prime}$ has the exact asymptotic decay of order $e^{-(n-1)r}$ at infinity \cite{BGGV}:
\begin{align*}
\lim_{d(0,x) \rightarrow \infty} e^{(n-1)d(0,x))}\, \calu^{\prime}(x) = -(n-1)\lim_{d(0,x) \rightarrow \infty} e^{(n-1)d(0,x)} \, \calu(x)
\end{align*}
and hence in particular $|\calu^{\prime}(x) |\leq C(n,p)\calu(x)$ for every $x \in \bn$ for some positive constant $C(n,p).$

According to \cite[Proposition 3.1, Corollary 4.3]{GMo}, for every $\tau_0 > 0,$ there exists a constant $A := A(n,m,\tau_0,u_0) $ such that the solution $w$ to \eqref{rFDE} satisfies
\begin{align*}
A^{-1} e^{-(n-1)r} \leq w(x,\tau) \leq Ae^{-(n-1)r}, \ \ \mbox{for all} \ (x, \tau) \in \bn \times [\tau_0,\infty).
\end{align*}
As a corollary, for a large time, $w$ remains positive inside $\bn.$ Regarding, the radial derivative in space, it has been proved that near the extinction time $T, u^{\prime}(x,t)$ uniformly decays like $(1-t/T)^{\frac{p}{p-1}} \, \calu^{p}(x)$ \cite[Theorem 1.1]{GMo}. This estimates translated to $w,$ and  can be written as: for $\tau_0$ large enough there exists a constant $A=A(n,p,\tau_0,u_0)$ such that 
\begin{align*}
\left|\left(\frac{w(x,\tau)}{\calu(x)}\right)^{\prime}\right| \leq A, \ \ \ \ \mbox{for all} \ \tau \geq \tau_0
\end{align*}
 where $\prime$ denotes the radial derivative.  Also note that by \eqref{conv_1fd}
\begin{align*}
w(\cdot, \tau) = \left(\frac{u(\cdot, \tau)}{\calu_m(\cdot, \tau)}\right)^m \calu \rightarrow \calu \ \mbox{in} \ L^{p+1}(\bn) \ 
\mbox{as} \ \tau \rightarrow \infty.
\end{align*}

\medskip

\noindent
{\bf Lyapunov functional and energy inequalities.} 
We recall the Lyapunov functional of the system \eqref{rFDE}
\begin{align*}
I_{0}(w)= \frac{1}{2}\int_{\bn}|\nabla_{\bn}w|^{2} \ \dvg -\frac{1}{p+1} \int_{\bn} w^{p+1} \ \dvg.
\end{align*}
In other words, the energy $I_0$ decays along the flow \eqref{rFDE}. Indeed a straightforward computation gives
\begin{align} \label{fdee1}
\frac{d}{d\tau} \int_{\bn} w^{p+1} \ \dvg &= \frac{p+1}{p}\int_{\bn} w \partial_{\tau} w^p \dvg \notag\\
& = \int_{\bn} w(\Delta_{\bn} w+ w^p) \ \dvg \notag\\
& = \frac{p+1}{p} \int_{\bn} (w^{p+1} - |\nabla_{\bn} w|^2) \ \dvg,
\end{align}
and 
\begin{align} \label{fdee2}
\frac{d}{d\tau} \int_{\bn} |\nabla w|^2 \ \dvg &= 2 \int_{\bn} \langle \nabla_{\bn} w, \nabla_{\bn} \partial_{\tau} w \rangle_{\bn} \ \dvg \notag\\
& = -\frac{2}{p} \int_{\bn} \frac{\Delta_{\bn}w}{w^{p-1}}(\partial_{\tau} w^p) \ \dvg \notag \\
& = -\frac{2}{p} \int_{\bn} \frac{\Delta_{\bn}w}{w^{p-1}}(\Delta_{\bn}w + w^p) \ \dvg \notag\\
& = \frac{2}{p} \int_{\bn} \left(|\nabla w|^2 - \frac{(\Delta_{\bn}w)^2}{w^{p-1}}\right) \ \dvg \notag\\
& = \frac{2}{p} \int_{\bn} \left(w^{p+1} - |\nabla w|^2 - \frac{(\Delta_{\bn} w + w^p)^2}{w^{p-1}}\right) \ \dvg.
\end{align}

We deduce from \eqref{fdee1}, \eqref{fdee2} 
\begin{align*}
\frac{d}{d \tau} I_0(w) = -\frac{1}{p} \int_{\bn} \frac{(\Delta_{\bn} w + w^p)^2}{w^{p-1}} \ \dvg \leq 0,
\end{align*}
and hence $I_0(w)$ decreases along the flow as $\tau$ increases. In particular, $I_0(w(\tau)) - I_0(\calu)$ decreases as $\tau$ increases and $I_0(\calu) = \frac{p-1}{2(p+1)}S_{n,p,0}(\bn)^{\frac{p+1}{p-1}}$.

\medskip

\noindent
{\bf More notations:}  We set 
\begin{align*}
\delta(\tau) = \|I_0^{\prime}(w(\tau))\|_{H^{-1}(\bn)}, \ \ y(\tau) = I_0(w(\tau)) - I_0(\calu).
\end{align*}
\subsection{Relative Linear Entropy} We  define the relative (linear) entropy of the system \eqref{rFDE}
\begin{align*}
E[w(\tau)] := \int_{\bn} |w(x,\tau)-\calu(x)|^2\calu^{p-1}(x) \ \dvg(x).
\end{align*}
Our first result asserts that the relative (linear) entropy of $w(\tau)$ decays exponentially as $\tau$ approaches $\infty.$

\medskip

\begin{theorem}\label{entropy decay}
Let $n\geq 3,$ $m,u_0$ satisfies the hypothesis {\bf (H2)} and let $w: \bn \times [0,\infty) \rightarrow [0,\infty)$ be the radial solution to \eqref{rFDE}. Then there exists a constant $\kappa = \kappa(n,m) >0$ independent of the initial data and a constant $C = C(n,m,u_0) >0$ such that 
\begin{align*}
\int_{\bn} |w(\tau)- \calu|^{p+1} \ \dvg \leq Ce^{-\kappa \tau}, \ \ \mbox{for all} \ \tau >0 \ \mbox{large}.
\end{align*}
In particular, the relative (linear) entropy decays exponentially in $\tau$
\begin{align*}
E[w(\tau)] \leq Ce^{-\frac{2\kappa}{p+1}\tau}, \ \ \mbox{for all} \ \tau >0 \ \mbox{large}.
\end{align*}

\end{theorem}


\begin{proof}
We divide the proof into the following small steps:

\medskip 

\noindent
{\bf Step 1:} By H\"older inequality with exponent $\frac{p+1}{2p} + \frac{p-1}{2p} = 1$ we deduce
\begin{align*}
\int_{\bn} |\Delta_{\bn} w + w^p|^{\frac{p+1}{p}} \ \dvg\leq \left(\int_{\bn} \frac{(\Delta_{\bn} w + w^p)^2}{w^{p-1}} \ \dvg\right)^{\frac{p+1}{2p}} \left(\int_{\bn} w^{p+1} \ \dvg\right)^{\frac{p-1}{2p}}.
\end{align*}
Since $1 < p < \tstar -1, $ by Poincar\'e-Sobolev inequality, $L^{\frac{p+1}{p}}(\bn) \subset H^{-1}(\bn)$ and hence 
\begin{align*}
\int_{\bn} \frac{(\Delta_{\bn} w + w^p)^2}{w^{p-1}} \ \dvg \geq \|I_0^{\prime}(w)\|_{H^{-1}(\bn)}^2 \left(\int_{\bn}w^{p+1}\right)^{-\frac{p-1}{p+1}}.
\end{align*}
Since $w(\cdot, \tau) \rightarrow \calu$ in $L^{p+1}(\bn),$ as $\tau \rightarrow \infty$ for $\tau$ large enough $\|w\|_{L^{p+1}(\bn)} \sim \|\calu\|_{L^{p+1}(\bn)} = S_{n,p,0}(\bn)^{\frac{1}{p-1}}$ and hence there exists a constant $C(n,p)>0$ such that 
\begin{align}\label{di0}
\frac{d}{d \tau} I_{0}(w(\tau)) \leq -C(n,p) \|I_0^{\prime}(w)\|_{H^{-1}(\bn)}^2 = -C(n,p)\delta(\tau)^2,
\end{align}
for $\tau \geq \tau_0$ and $\tau_0$ large enough. We emphasize that the constant $C$ does not depend on the initial datum.

\medskip

\noindent
{\bf Step 2:}  $\lim_{\tau \rightarrow \infty} I_0(w(\tau))$ exists and equal to $I_0(\calu).$

\medskip

 Again using $w(\cdot, \tau) \rightarrow \calu$ in $L^{p+1}(\bn),$ as $\tau \rightarrow \infty,$ we conclude $I_0(w(\tau))$ remains bounded from bellow for all time. As a result, from \eqref{di0} we deduce $\int_{\tau_0}^{\infty} \delta(\tau)^2 \ d\tau < \infty$ and hence there exists a sequence $\tau_k$ such that $ \|I_0^{\prime}(w(\tau_k))\|_{H^{-1}(\bn)} \rightarrow 0.$ According to the results of \cite{BS}, 
$w(\tau_k)$ converges in $H^1(\bn)$ to the sum of bubbles of the form $\calu_{b}, b \in \bn$ and in particular, by Poincar\'e-Sobolev inequality, in $L^{p+1}(\bn)$-norm. Hence we conclude $w(\tau_k) \rightarrow \calu$ in $H^1(\bn)$  as $k \rightarrow \infty$. As a result,

\begin{align*}
\lim_{k \rightarrow \infty} I_0(w(\tau_k)) = I_0(\calu) = \frac{p-1}{2(p+1)}S_{n,p,0}(\bn)^{\frac{p+1}{p-1}}.
\end{align*}

Since $I_0(w(\tau))$ decreases as $\tau$ increases we conclude $\lim_{\tau \rightarrow \infty} I_0(w(\tau))=I_0(\calu).$

\medskip

\noindent
{\bf Step 3.} Conclusion: exponential decay of the (linear) entropy.
\medskip

Using Step $2,$ and Corollary~\ref{corthm2}, we conclude
\begin{align*}
\inf_{b \in \bn} \int_{\bn}|\nabla_{\bn}(w(\tau) - \calu_b)|^2 \ \dvg \leq C(n,p) \|I_0^{\prime}(w(\tau))\|_{H^{-1}(\bn)}^2
\end{align*}
holds for some constant $C(n,p)>0$ independent of $\tau.$ Let $\calu_{b(\tau)}$ be the minimizer and set $v(\tau) = w(\tau) - \calu_{b(\tau)}.$ Since $w(\tau) \geq 0,$ we get $v(\tau) /\calu_{b(\tau)} \geq -1$ and therefore by Bernoulli's inequality $(\calu_{b(\tau)} + v(\tau))^{p+1} - \calu_{b(\tau)}^{p+1} - (p+1)\calu_{b(\tau)}^pv(\tau) \geq 0.$ Moreover, all the $\calu_{b(\tau)}$ has same energy as $\calu: I_{0}(\calu_{b(\tau)}) = I_0(\calu).$ Now
\begin{align}\label{di1}
I_0(w(\tau)) &= \frac{1}{2}\int_{\bn}|\nabla_{\bn}(\calu_{b(\tau)} + v(\tau))|^{2} \ \dvg -\frac{1}{p+1} \int_{\bn} (\calu_{b(\tau)} + v(\tau))^{p+1} \ \dvg \notag\\
&= I_0(\calu_{b(\tau)})+  \frac{1}{2}\int_{\bn}|\nabla_{\bn} v(\tau)|^2 \ \dvg+\int_{\bn}
\langle \nabla \calu_{b(\tau)}, \nabla v(\tau)\rangle_{\bn}  \ \dvg \notag\\
& \ \ \ \ \ \ \ \ \ \ \ \ \ - \frac{1}{p+1}\int_{\bn}\left[ (\calu_{b(\tau)} + v(\tau))^{p+1} - \calu_{b(\tau)}^{p+1}\right]\ \dvg \notag\\
&= I_0(\calu)+  \frac{1}{2}\int_{\bn}|\nabla_{\bn} v(\tau)|^2 \ \dvg + \int_{\bn}
\calu_{b(\tau)}^p v(\tau) \ \dvg \notag\\
 &  \ \ \ \ \ \ \ \ \ \ \ \ \ - \frac{1}{p+1}\int_{\bn}\left[ (\calu_{b(\tau)} + v(\tau))^{p+1} - \calu_{b(\tau)}^{p+1}\right]\ \dvg \notag\\
 & = I_0(\calu)+  \frac{1}{2}\int_{\bn}|\nabla_{\bn} v(\tau)|^2 \ \dvg \notag\\
 &  \ \ \ \ \ \ \ \ \ \ \ \ \ - \frac{1}{p+1}\int_{\bn}\left[ (\calu_{b(\tau)} + v(\tau))^{p+1} - \calu_{b(\tau)}^{p+1} - (p+1)\calu_{b(\tau)}^p v(\tau)\right]\ \dvg \notag\\
 &\leq I_0(\calu) + C(n,p) \|I_0^{\prime}(w(\tau))\|_{H^{-1}(\bn)}^2.
\end{align}

Recalling $y(\tau) = I_{0}(w(\tau)) -  I_0(\calu)$ decreases as $\tau$ increases and $\lim_{\tau \rightarrow \infty} y(\tau) = 0,$  we deduce $y(\tau) \geq 0$ for $\tau$ large. Moreover, from \eqref{di0} and \eqref{di1} we see that 
 \begin{align*}
 y^{\prime}(\tau) \leq -C(n,p) y(\tau),
 \end{align*}
from which we conclude $0 \leq y(\tau) \leq Ce^{-\kappa_0 \tau},$ where $\kappa_0 = \kappa_0(n,p)>0$ is a constant 
independent of the initial datum $u_0.$

Moreover, integrating \eqref{di0} from $\tau$ large enough to $\infty$ we get
\begin{align*}
\int_{\tau}^{\infty} \delta(\hat\tau)^2  d\hat\tau \leq I_0(w(\tau)) - I_0(\calu) \leq Ce^{-\kappa_0\tau}.
\end{align*}

 As a result, by H\"older inequality
 \begin{align*}
 \int_{\tau}^{\infty} \delta(\hat\tau) d\hat\tau &= \sum_{k=0}^{\infty} \int_{k + \tau}^{k+1+\tau}\delta(\hat\tau) d\hat\tau \\
 &\leq \sum_{k=0}^{\infty} \left(\int_{k + \tau}^{k+1+\tau}\delta(\hat\tau)^2 d\hat\tau\right)^{\frac{1}{2}} \\
 &\leq \sum_{k=0}^{\infty} (Ce^{-\kappa_0(\tau+k)})^{\frac{1}{2}} \leq C_1e^{-\kappa \tau}
 \end{align*}
where $\kappa = \kappa_0/2.$ Now applying $|a-b|^{p+1} \leq |a^{p+1} - b^{p+1}| $ for $a,b \geq 0$ and using the equation \eqref{rFDE} satisfied by $w$ we get for $\tau_1 > \tau$
\begin{align*}
\int_{\bn}|w(x,\tau_1) - w(x,\tau)|^{p+1}\dvg(x) &\leq \int_{\bn}(|w(\tau_1)| + w(\tau)||w(\tau_1)^p - w(\tau)^p| \dvg \\
&= \int_{\bn} (|w(\tau_1)| + |w(\tau)|) |\int_\tau^{\tau_1} \frac{d}{d\hat\tau} w^p(\hat\tau)| \dvg\\
&= \int_{\tau}^{\tau_1} \int_{\bn}(|w(\tau_1)| + |w(\tau)|)(\Delta_{\bn}w(\hat\tau) + w^p(\hat\tau)) \dvg\\
&\leq C \int_{\tau}^{\tau_1} \|\Delta_{\bn}w(\hat\tau) + w^p(\hat\tau)\|_{H^{-1}(\bn)} \ d\hat \tau \\
&\leq C \int_{\tau}^{\infty} \delta(\hat\tau) \ d\hat \tau\leq C e^{-\kappa \tau}.
\end{align*}
Letting $\tau_1\to\infty$, we get the desired convergence stated is exponential:
\begin{align}\label{entropy related}
\|w(\tau)-\calu\|_{L^{p+1}(\bn)}^{p+1}\leq C e^{-\kappa\tau}.
\end{align}
The decay of the (linear) entropy now follows from \eqref{entropy related} and H\"{o}lder inequality.
\medskip

 \end{proof}

\begin{remark}
{\rm 
As an immediate remark,  the exponential decay of $\int_{\bn}|w(\tau) - \calu|^{p+1} \ \dvg,$ can be translated to exponential decay of  $\| w(\tau) - \calu\|_{L^{\infty}(\bn)}$ in $\tau,$ using standard interpolation inequalities.  As a result, we get exponential decay of the relative error in \it local uniform topology\rm. However, the hyperbolic space being infinite volume and lack of compactness even in the subcritical case,  getting exponential decay in uniform topology is a very technically challenging issue and will be addressed in the next section. 
}
\end{remark}
\medskip
\begin{remark}
{\rm 
In the original variables, the results of Theorem \ref{entropy decay} translates into the following: Let $u_0,m,p$
as in hypothesis {\bf (H2)} and let $u$ be the (radial) solution to \eqref{FDE} with initial data $u_0$ and let $T$ be the extinction time. Then there exists a constant $\hat \kappa =\hat \kappa(n,m)$ and  $C = C(n,m,u_0)$ such that 

\begin{align*}
\int_{\bn} \left| \frac{u(x,t)^m}{\calu_m(x,t)^m} - 1\right|^2 \calu^{p+1} \ \dvg \leq C\left(\frac{T-t}{T}\right)^{\hat\kappa T},
\end{align*}
for all $t \in [T_0,T],$ where $\calu_m$ is defined by \eqref{calum}. Indeed,   using \eqref{entropy related} and the relation between $\tau$ and $t$, i.e., $\tau = \frac{T}{1-m}\ln \left(\frac{T}{T-t}\right)$, we have
\begin{align*}
\int_{\bn} \left| \frac{u(x,t)^m}{\calu_m(x,t)^m} - 1\right|^2 \calu^{p+1} \ \dvg &=\int_{\bn}\left| w(x,\tau) - \calu\right|^2 \calu^{p-1}\dvg\\
&\leq \left(\int_{\bn}| w(x,\tau) - \calu|^{p+1}\dvg\right)^\frac{2}{p+1}\left(\int_{\bn}\calu^{p+1}\dvg\right)^\frac{p-1}{p+1}\\
&\leq Ce^{-\frac{2\kappa}{p+1}\tau}=C\left(\frac{T-t}{T}\right)^{\hat\kappa T},
\end{align*}
where $\hat\kappa=\frac{2\kappa}{(1-m)(p+1)}$ and  $\hat\kappa$  depends on $n$ and $m$, since $\kappa$ depends only on them.
}

\end{remark}

\medskip

\subsection{Relative error}\label{re} We shall now derive the rate of convergence for the relative error. It can be seen from the decay of the relative (linear) entropy and ensuing
 the approach of Bonforte and Figalli \cite{FD7}, we can obtain the extinction rate in the relative uniform norm under some assumptions on $p.$ 
 Namely, we need the restriction $p>2.$ Also recall we are in the sub-critical range $p < \tstar-1,$ and hence the condition $p>2$ is void in dimension $n \geq 6.$ The technical reason behind this restriction seems to stems out of the fact that $\calu^{p-1}$ is integrable and $\int_{\bn}\calu^p G_{\bn}(x,y) \ \dvg $
decays of the order of $\calu(x)$ as $d(x,0) \rightarrow \infty$ provided $p>2,$ where $G_{\bn}$ is the Green's function of $-\Delta_{\bn}.$ We do not know yet whether this restriction is purely technical and will be the subject of future studies. Below we state our result of the extinction rate in dimension $3 \leq n \leq 5.$

\medskip

\begin{theorem}\label{asymptotic}
Let $3 \leq n\leq 5,$ $m \in (m_s, \frac{1}{2}),u_0$ satisfies the hypothesis {\bf (H2)} and let $u: \bn \times [0,\infty) \rightarrow [0,\infty)$ be the radial solution to \eqref{FDE}. Let $T := T(u_0) \in (0,\infty)$ be the finite time extinction of $u,$ then there exists a constant $\mu = \mu(n,m) >0$ independent of the initial data and a constant $C = C(n,m,u_0) >0$
such that 

\begin{align*}
\left \| \frac{u(\cdot, t)^m}{\calu_m(\cdot,t)^m} - 1\right\|_{L^{\infty}(\bn)} \leq C\left(\frac{T-t}{T}\right)^{\mu T} \ \mbox{for all} \ t \in (T_0,T).
\end{align*}

\end{theorem}

The proof is relatively long, quite intriguing and will be proven in the next section.


 \section{Smoothing effect for the relative error and rate of convergence}\label{relative-error-theorem}
 
 This section is devoted to proving the smoothing estimates for the relative error 
 
 \begin{equation}\label{error}
 v(x,\tau) : = \dfrac{w(x,\tau)}{\calu(x)} - 1 := \dfrac{f(\tau)}{\calu},
 \end{equation}
 where $w$ is a solution to the \eqref{rFDE} and $f(\tau) = w(\tau) - \calu$. We already know that $v(\tau) \rightarrow 0$ in $L^{\infty}(\bn)$ norm as $\tau \rightarrow \infty$ \cite{GMo}. Moreover, thanks to the decay estimates mention in Section \ref{fd sec}, we have the following bound: for $\tau_0$ large enough, there exists a constant $C$ depending on $n,m,\tau_0,u_0$ such that
 \begin{align} \label{decayvprime}
 |v(x,\tau)| + |v^{\prime}(x,\tau)| \leq C, \ \ \ \ \mbox{for all} \ x \in \bn \ \mbox{and} \ \tau \geq \tau_0,
 \end{align}
where $\prime$ denotes the radial differentiation in the space variable. Our main objective in this section is to derive 
 smoothing estimates for $v$ in terms of the relative (linear) entropy. In particular, we prove the following: 
 
 \subsection{Growth estimates for the relative error}
 
 \begin{theorem} \label{linftyentropy}
 Let $v$ be the relative error defined in \eqref{error}. Then there exists a constant $\bar\kappa = \bar\kappa(n,m)$ such that the following estimates hold for any $\tau_1 \geq \tau_0$ with $\tau_0$ large enough and $\tau_1 - \tau_0 \leq 1:$
 
 \begin{equation*}
  \|v(x,\tau_1) \|_{L^{\infty}(\bn)} \, \leq \, K_1 \, \dfrac{e^{2m(\tau_1 - \tau_0)}}{\tau_1 - \tau_0} \left(  \sup_{s \in [\tau_0, 2\tau_1-\tau_0]} E[w(s)] \right)^{\bar\kappa} \, + \, 2m (\tau_1 - \tau_0) e^{2m(\tau_1 - \tau_0)},
  \end{equation*}
where $K_1 > 0$ is a constant depends on $n, p.$ 
 \end{theorem}
 
 The proof of the above theorem rests on the following two key lemmas. Let us first state the lemma. 
 
 \begin{lemma} \label{intermediate}
 Let $T > 0$ be the extinction time of the solution $u$ to \eqref{FDE} and let $v$ be the relative error defined by \eqref{error}. Then the following estimates hold for any $\tau_1 \geq \tau_0 \geq \frac{T}{1-m} \, \log 2$ and for 
 each $x \in \bn :$
 
 \begin{align}\label{integral-esti}
& \frac{T}{2m} \left( 1 - e^{-\frac{2m}{T} (\tau_1 - \tau_0)} \right) v(x,\tau_1) \, - \,  \frac{m}{T} (\tau_1 - \tau_0)^2 \notag \\ \leq & \int_{\tau_0}^{\tau_1} v(x,\tau) \, {\rm d}\tau \notag \\
  \leq &\frac{T}{2m} \left( e^{\frac{2m}{T} (\tau_1 - \tau_0)} - 1 \right) v(x,\tau_0) \, + \, \frac{m}{T} (\tau_1 - \tau_0)^2 e^{\frac{2m}{T} (\tau_1 - \tau_0)}.
 \end{align}
 \end{lemma}
 
 \begin{proof}
Essentially our arguments rely on the celebrated Benilan-Crandall estimate on the solution $u,$ which is stated as time monotonicity estimates for $u$, i.e., 
  $u_{t} \leq \frac{u}{(1 - m)t},$ which  follows from from the fact that $t \mapsto t^{-\frac{1}{(1-m)}} u(x,t)$ is nonincreasing in time for a.e. $x \in \bn$
  and for all $t > 0.$ Indeed, consider the rescaled solution $u_{\lambda}(x, t) \, = \,  \lambda^{-\frac{1}{(1-m)}}  \, u(x, \lambda t),$ then 
  $u_{\lambda} (\cdot,0) \, = \, \lambda^{-\frac{1}{(1-m)}} u_0.$ Now by choosing $\lambda = \frac{h + t}{t} \geq 1,$ we obtain for all $h \geq 0$
  \begin{align*}
  0 \leq u(x, t) \, - \, u_{\lambda}(x, t) = u(x, t) \, - \,  \lambda^{-\frac{1}{(1-m)}} u(x, \lambda t) \,  = \,  u(x, t) \, -\, \left( \frac{h + t}{t} \right)^{-\frac{1}{(1-m)}} u(x,  h + t) 
  \end{align*}
 The first inequality in the above follows from comparison principle and $u_{\lambda} (\cdot,0) = \frac{u_0}{\lambda^{\frac{1}{(1-m)}}} \, \leq \, u_0.$ This proves the desired 
 monotonicity of the  function $t \mapsto t^{-\frac{1}{(1-m)}} u(x,t).$ Now proceeding as in \cite[Lemma~{4.3}]{FD7}, we prove the desired estimate
  \eqref{integral-esti}. For the convenience of the readers, we sketch a proof of the steps. The  Benilan-Crandall estimate of $u$ translates to $w$ is as follows:
  
  \begin{align*}
 \frac{\partial_{\tau} w}{w} \leq \frac{2m}{T} \qquad\forall\, \tau \geq \frac{T}{1-m} \log 2.
  \end{align*}
 As a result, the relative error $w$ satisfies $\partial_{\tau} v \leq \frac{2m}{T}(v+1)$ for all $\tau \geq \frac{T}{1-m} \log 2.$ Integrating from $\tau_0$ to $\tau \geq \tau_0$ we get
 
 \begin{align} \label{for v}
 v(\cdot,\tau) \leq v(\cdot,\tau_0)e^{\frac{2m}{T} (\tau - \tau_0)} + \left( e^{\frac{2m}{T} (\tau - \tau_0)} - 1\right), \ \ \tau \geq \tau_0.
 \end{align}
Now for $\tau_1 \geq \tau_0 \geq \frac{T}{1-m} \log 2, \, \tau \in [\tau_0, \tau_1]$ and using $e^s-1 \leq se^s$ for $s\geq 0$ we get from \eqref{for v}
\begin{align*}
 v(\cdot,\tau) \ \leq \ v(\cdot,\tau_0)e^{\frac{2m}{T} (\tau - \tau_0)} + \frac{2m}{T} (\tau - \tau_0)e^{\frac{2m}{T} (\tau - \tau_0)} 
 \end{align*}
Further, using the inequality $e^{-s}\geq 1-s$ for $s\geq 0$, we get from \eqref{for v} that
 \begin{align*}
v (\cdot,\tau) \ \geq \ v(\cdot,\tau_1)e^{-\frac{2m}{T} (\tau_1 - \tau)} - \frac{2m}{T} (\tau_1 - \tau).
\end{align*}
Integrating the last two inequalities from $\tau_0$ to $\tau_1$ and noticing $e^{\frac{2m}{T} (\tau - \tau_0)} \leq e^{\frac{2m}{T} (\tau_1 - \tau_0)}$ we get \eqref{integral-esti}.
 \end{proof}
 
 \medskip 
 
 In order to bound $\int_{\tau_0}^{\tau_1} v(\tau) \ d\tau$ in terms of the relative (linear) entropy, we need to impose an assumptions on the exponent $p$ which in turn will impose a dimensional restriction $n \in [3,5]$ as we  have discussed in the beginning of Section~\ref{re}. 
 \medskip
 
 \begin{lemma}\label{second-lemma}
 Let $T > 0$ be the extinction time of the solution $u$  to \eqref{FDE} and let $v$ be the relative error defined by \eqref{error}. Further assume that $3 \leq n \leq 5$ and $2 < p < \tstar -1.$ Then there exists a constant $\bar\kappa = \bar\kappa(n,m)$ such that the following estimate holds true for any $\tau_1 \geq \tau_0$ large enough and for 
 every $x \in \bn :$
 
 \begin{equation*}
 \left| \int_{\tau_0}^{\tau_1} v(x, \tau) \, {\rm d}t \right| \, \leq \, \left( K_1 \, + \, K_2 (\tau_1 - \tau_0) \right) \left( \sup_{\tau \in [\tau_0, \tau_1]} E [w(\tau)] \right)^{\bar\kappa}, 
 \end{equation*}
 where the constants $K_1, K_2$ depends on $n, p.$
 \end{lemma}
 
 \begin{proof}
 To prove the lemma, we partly follow and adapt the ideas of Bonforte-Figalli \cite[Lemma~4.4]{FD4} in our setting.
  Although due to the non-compactness of the hyperbolic space, it requires a new way to handle the integral involving the relative error.  The proof is divided into several steps. 
 
 \medskip 
 
 \noindent
 {\bf Step 1:} In this step, we shall find the {\it dual \rm} equation satisfied by the relative error. Since  $w \, = \, (v + 1) \, \calu$ a simple computation gives
 \begin{align*}
 -\Delta_{\bn} ((v + 1) \, \calu) = - \Delta_{\bn} (v \, \calu) \, - \, \Delta_{\bn} \calu \, = \, - \partial_{\tau} \, w^p \, + \, w^p.
 \end{align*}
 or equivalently, 
 \begin{equation}\label{equivalently-eq}
 v(\cdot, \tau) \, \calu \, = \, (-\Delta_{\bn})^{-1} \left( - \partial_{\tau} \, w^p \, + \, (w^p \, - \, \calu^p) \right).
 \end{equation}
 Therefore using the Green function representation we can write further 
 
  \begin{equation}\label{equivalently-eq}
  v(., \tau) \, \calu \, = \, - \int_{\bn} ( \partial_{\tau} \, w^p ) \, G_{\bn} (x, y) \, {\rm d}v_{\bn}(y) \, + \, \int_{\bn} (w^p \, - \, \calu^p) \, G_{\bn} (x, y) \, {\rm d}v_{\bn}(y),
  \end{equation}
 where $G_{\bn}(x, y)$ denotes the Green of $-\Delta_{\bn}.$ In the hyperbolic space, the Green function is 
  given by $G_{\bn}(x, y) := \int_{r}^{\infty} (\sinh s)^{-(n-1)} \, {\rm d}s$ and $r := d(x, y)$ denotes the hyperbolic distance between the points $x,y \in \bn$.  Hence the following sharp estimates of the Green function 
  $G_{\bn}$ holds for $n \geq 3$ (see also \cite[Eq~17]{RSZ}) :
   \begin{align} \label{green-function}
  G_{\bn} (x, y) \asymp
\begin{cases}
  (d(x, y))^{-(n-2)}, \quad  \mbox{for all} \ (x, y) \in \bn :  \ 0 < d(x, y) \leq 1, \\
  e^{-(n-1) d(x, y)}, \quad \quad \mbox{for all} \ (x, y) \in \bn :   d(x, y) > 1.
\end{cases}
\end{align}
Here $f \asymp g$ denotes they are comparable namely, there exist positive constants $C_1, C_2,$ depending only on $n,$
 such that $C_1 \, g(x) \leq f(x) \leq C_2 \, g(x)$ for all $x \in \bn$. 
 
 Now integrating \eqref{equivalently-eq} over $(\tau_0, \tau_1)$ we obtain
 
 \begin{align*}
 \calu (x) \, \int_{\tau_0}^{\tau_1} v(x, \tau) \, {\rm d}\tau  \, & = \,  \underbrace{\int_{\bn} \left( w^p(y, \tau_0) \, - \,  w^p(y, \tau_1) \right) \, G_{\bn}(x, y) \, \dvg(y)}_{:= I_1} \notag \\
 & + \underbrace{\int_{\tau_0}^{\tau_1} \int_{\bn} \left( w^p(y, \tau) \, - \, \calu^p(y) \right) \, G_{\bn}(x, y) \, \dvg(y)}_{:= I_2}.
 \end{align*}
  Next, we estimate $I_1$ and $I_2$ separately. 
  
  \medskip 
  
  \noindent
  {\bf Step 2 :} In this step,  we collect some straightforward estimates using the properties of  $\calu$ and $w$ stated in Section \ref{fd sec}. We first recall the inequality 
  $|a^p \, - \, b^p| \leq p \, \mbox{max}( a^{p-1}, b^{p-1}) |a - b|,$ which holds for all $a, b > 0$ and $p \geq 1.$ Moreover, using the fact that $v(\tau) \rightarrow 0$  we establish 
  $\frac{1}{2} \, \calu \leq w \leq \frac{3}{2} \, \calu.$ Therefore we have the following estimates

  \begin{align}\label{error-1}
  \left| w^p(y, \tau_0) \, - \, w^p (y, \tau_1) \right| & \leq p \left(\frac{3}{2} \right)^{p-1} \, \calu^{p-1} \left| w(y, \tau_0) \, - \, w(y, \tau_1) \right| \notag \\
  & := K_1 \, \calu^p \left| v(y, \tau_0) \, - \, v(y, \tau_1) \right| \, \leq 3K_1 \, \calu^p.
  \end{align} 
  Analogously, 
  
  \begin{align}\label{error-2}
  \left| w^p(y, \tau) \, - \, \calu^p(y) \right| \leq K_1 \, \calu^{p} \, \left| v(y, \tau) \right| \leq 4K_1 \, \calu^p.
  \end{align}
 Therefore in terms of linear Entropy, we can write 
 
 \begin{align*}
 \int_{\bn} v(x,\tau)^2 \, \calu^{p+1}(x) \, \dvg \, = \, \int_{\bn} (w(x,\tau) \, - \, \calu(x))^2 \, \calu^{p-1} \, \dvg \, =\, E[w(\tau)] .
 \end{align*}
  As a consequence of the above, we can write 
  
  \begin{align*}
  \int_{\bn} (v(y, \tau_0) \, - \, v(y, \tau_1))^2 \, \calu^{p+1} \, \dvg & \leq 2\int_{\bn} (v(y, \tau_0))^2 \, \calu^{p+1} \, \dvg \, + \, 2 \int_{\bn} (v(y, \tau_1))^2 \, \calu^{p+1} \, \dvg \\
  & \leq K_2 \,  \sup_{\tau \in [\tau_0, \tau_1]} E[w(\tau)] .
  \end{align*}
  Using H\"older inequaliity and assuming $p > 2,$ we obtain 
 
  \begin{align*}
  \int_{\bn} |w(y,\tau) \, - \, \calu(y)| \, \calu(x)^p \, \dvg &\leq \left( \int_{\bn} |w(y,\tau) \, - \, \calu(y)|^2 \, \calu^{p+1}(y) \, \dvg \right)^{\frac{1}{2}} \underbrace{\left( \int_{\bn} \calu^{p-1} \, \dvg
  \right)^{\frac{1}{2}}}_{\leq C, \quad  \mbox{for} \ p > 2} \notag \\
  & \leq K_2 \, E[w(\tau)]^{\frac{1}{2}}.
  \end{align*}
  A similar argument gives us
  \begin{align*}
 \int_{\bn} \left| v(y,\tau_0) \, - \, v(y,\tau_1) \right| \, \calu^{p} \, \dvg(y) \leq K_2 \, \sup_{\tau \in [\tau_0,\tau_1]}E[w(\tau)]^{\frac{1}{2}}
  \end{align*}
 
   for $p >2.$
   
  \medskip 
  
  \noindent
  {\bf Step 3: Estimate of $I_1$ and $I_2$.} First, we estimate $I_1.$ Let $1< r < \frac{d(x, 0)}{2}$ to be fixed later, and $d(x,0)$ large enough
  
  \begin{align*}
  I_1 &:= \underbrace{\int_{B_r(0)} \left| w^p(y, \tau_0) \, - \,  w^p(y, \tau_1) \right| \, G_{\bn}(x, y) \, \dvg(y)}_{:= I_1^1} \, \\
  & + \, \underbrace{ \int_{B_r(x)} \left| w^p(y, \tau_0) \, - \,  w^p(y, \tau_1) \right| \, G_{\bn}(x, y) \, \dvg(y)}_{:= I_1^2} \, \\
  & + \underbrace{ \int_{\bn \setminus (B_r(0) \cup B_r(x))} \left| w^p(y, \tau_0) \, - \,  w^p(y, \tau_1) \right| \, G_{\bn}(x, y)}_{:= I_1^3} \, \dvg(y). \, 
  \end{align*}
 Now we estimate seperately each of $I_1^i,$ for $i  =1, 2, 3.$ 
 
 \medskip
 
 \noindent
 {\bf Estimate of $I_1^1:$} In this domain $d(x,y) \geq r >1$ and hence using \eqref{green-function} and estimates on $\calu$ we  get

 \begin{align*}
 | I_1^1| \leq K_1 \, \int_{B_r(0)} \left| w^p(y, \tau_0) \, - \,  w^p(y, \tau_1) \right| \, e^{-(n-1)d(x, y)} \, \dvg(y).
 \end{align*} 
  In $B_r(0)$, $d(x, y) \geq r,\,  d(y, 0) \leq r, \  \mbox{and} \  d(x, 0) > 2r $ and also we have 
   \begin{align*}
   d(x, y) \geq d(x, 0) - d(y, 0).
   \end{align*}
Using this we further estimate  $I_1^1$ 
   \begin{align}\label{esti-I_1^1}
  | I_1^1 | &\leq K_1 \, \int_{B_r(0)} \left| v^p(y, \tau_0) \, - \,  v^p(y, \tau_1) \right| \, \calu^{p}(y) \underbrace{e^{-(n-1)d(x, 0)}}_{: \asymp \, \calu(x)} \, e^{(n-1)d(y, 0)} \dvg(y) \notag \\
  & \leq K_1 \, \calu (x) \int_{B_r(0)} \left| v^p(y, \tau_0) \, - \,  v^p(y, \tau_1) \right| \, \calu^{p}(y) \, e^{(n-1)d(y, 0)} \dvg(y) \notag \\
  & \leq K_1 \, e^{(n-1)r} \calu (x) \int_{B_r(0)} \left| v^p(y, \tau_0) \, - \,  v^p(y, \tau_1) \right| \, \calu^{p}(y) \, \dvg(y) \notag \\
  & \leq K_2 \, e^{(n-1)r} \calu (x)\sup_{\tau\in[\tau_0,\tau_1]} E[w(\tau)]^{\frac{1}{2}}.
  \end{align}
 
 \medskip
 
 \noindent
{\bf Estimate of $I_1^2$:} Using \eqref{green-function}, \eqref{error-1},  and estimate on $\calu.$

 \begin{align*}
 | I_1^2|  & \leq K_1 \, \int_{B_r(x)} \left| w^p(y, \tau_0) \, - \,  w^p(y, \tau_1) \right| \,  \, G_{\bn}(x, y) \, \dvg(y) \\
 & \leq K_1 \, \int_{B_r(x)} \calu^p(y) \, G_{\bn}(x, y) \, \dvg(y)  
 \end{align*} 
 In $B_r(x)$ we have 
   $d(x, y) \leq r$, $d(y, 0) \geq r,$ and $$d( y, 0) \geq d(x, 0) - d(x, y).$$ 
   
   Though not necessary to it's full generality let us estimate the integral of the Green's function over $B_r(x)$ for every $r:$ For $0< s \leq 1$ it holds $\sinh s \leq s \cosh 1$. Therefore,  for $0< s \leq 1$  
   \begin{align*}
\int_{B_s(x)}  G_{\bn}(x, y) \, \dvg(y)\leq C_1\int_0^s \frac{(\sinh t)^{n-1}}{t^{n-2}}dt\leq \int_0^s\frac{t^{n-1}}{t^{n-2}}\left(\frac{\sinh t}{t}\right)^{n-1}dt\leq C_1(\cosh 1)^{n-1}s^2,
   \end{align*}
   where $C_1>0$ is a constant depending of dimension $n.$ On the other hand, for $r \geq 1,$ we have 
 $\int_{B_r(x)}  G_{\bn}(x, y) \, \dvg(y)\leq C_2 \, r,$
  where $C_2>0$ is again a dimensional constant.  Combining the above two estimates we obtain $$ \int_{B_r(x)}  G_{\bn}(x, y) \, \dvg(y) \leq C\, r^2,$$ 
  for any $r >0$ and for some positive constant $C:= C(n).$
  Using the above estimates, we estimate $I_1^2$ further as below
  \begin{align}\label{esti-I_1^2}
    | I_1^2 | &\leq  K_2 \int_{B_r(x)} e^{-(n-1) d(y, 0)p} \, G_{\bn}(x, y) \, \dvg(y) \notag \\
    & \leq K_2 \int_{B_r(x)} e^{-(n-1) d(x, 0)p} \, e^{(n-1) d(x, y)p} \, G_{\bn}(x, y) \, \dvg(y) \notag \\
    & \leq K_2  \, \calu(x) \, e^{-(n-1)(p-1) d(x, 0)} \, e^{(n-1)p r} \underbrace{\int_{B_r(x)}  G_{\bn}(x, y) \, \dvg(y)}_{\leq \, C\, r^2} \notag \\
    & \leq K_2  \, \calu(x) \, e^{-(n-1)(p-1) 2r} \, e^{(n-1)p r} \, r^2 = K_2 \, \calu(x) \, e^{-(n-1)(p-2) r} \, r^2.
\end{align}

\medskip

\noindent
{\bf Estimate of $I_1^3:$} We again use \eqref{green-function}, \eqref{error-1}, and  estimates on $\calu.$

 \begin{align*}
 | I_1^3| & \leq  \, \int_{\bn \setminus (B_r(0) \cup B_r(x))} \left| w^p(y, \tau_0) \, - \,  w^p(y, \tau_1) \right|  \, G_{\bn}(x, y) \, \dvg(y) \\
 & \leq  K_1 \, \int_{\bn \setminus (B_r(0) \cup B_r(x))} \calu^p(y) \, G_{\bn}(x, y) \, \dvg(y).
 \end{align*} 
In $\bn \setminus (B_r(0) \cup B_r(x))$ we have 
   $d(x, y) \geq r >1,\,  d(y, 0) \geq r,  $ and also we have 
   $$
   d( x, y) \geq d(x, 0) - d( y, 0).
   $$

  Hence using this we estimate  $I_1^3$ further 

 \begin{align}\label{esti-I_1^3}
  | I_1^3| & \leq K_2 \int_{\bn \setminus (B_r(0) \cup B_r(x))} \calu^{p} \underbrace{e^{-(n-1) d(x, 0)}}_{\asymp \, \calu(x)} \underbrace{e^{(n-1)d(y, 0)}}_{\asymp \, (\calu (y))^{-1}} \, \dvg(y) \notag \\
  & \leq K_2 \, \calu(x) \int_{\bn \setminus (B_r(0) \cup B_r(x))} \calu^{p-1}(y) \, \dvg(y) \notag \\
  & \leq K_2 \, \calu(x) \int_{\{d(0,y)>r\}} \calu^{p-1}(y) \, \dvg(y) \leq K_2 \, \calu(x) \, e^{-(n-1) (p-2)r}.
   \end{align}

Now combining \eqref{esti-I_1^1}, \eqref{esti-I_1^2} and \eqref{esti-I_1^3}, we obtain 

\begin{align*}
\left| I_1\right| & \leq K_1\,  \calu(x) e^{-(n-1) (p-2) r} \,r^2 \; + \, K_2\;  \calu(x) e^{(n-1)r} \, E[w(\tau)]^{\frac{1}{2}} \; + \; K_3 \, \calu(x) e^{-(n-1) (p-2)r}, \notag \\
& \leq \calu(x)\left(K_1 \, e^{-(n-1) (p -2 - \delta) r} \, + \, K_2 \, E[w(\tau)]^{\frac{1}{2}} e^{(n-1)r}\right).
\end{align*}
where constants $K_1, K_2, K_3$ are depending only on $n, p,$ and  for $\delta >0.$ 



 \medskip 

The estimate of $I_2$  essentially follows the same steps as in step~3. Here we shall use the 
estimate \eqref{error-2} instead of \eqref{error-1} to conclude 

\begin{align*}
&\int_{\bn} \left| w^p(y, \tau) \, - \, \calu^p(y) \right| \, G_{\bn}(x, y) \, \dvg(y) \\ 
\leq \ \ \ &\calu(x)\left(K_1 \, e^{-(n-1) (p -2 - \delta) r} \, + \, K_2 \, E[w(\tau)]^{\frac{1}{2}} e^{(n-1)r}\right)
\end{align*}
where $1< r < \frac{d(x, 0)}{2}.$ Hence integrating over $(\tau_0, \tau_1)$ we obtain 

\begin{equation*}
\left| I_2\right| \leq \, \calu(x) \, |\tau_1 - \tau_0| \left(K_1 \, e^{-(n-1) (p -2 - \delta) r} \, + \, K_2 \,  e^{(n-1)r}\left(\sup_{\tau \in [\tau_0,\tau_1]}E[w(\tau)]\right)^{\frac{1}{2}}\right).
\end{equation*}

Note that this is not enough to conclude our result; hence, we need an interior estimate.

\medskip

\noindent
{\bf Step 4.} We fix $r_0$ large and to be chosen later. For $r \leq r_0$ we have 
\begin{align*}
v^2(r,\cdot)\calu(r) &= \int_{r}^{\infty} \frac{d}{ds}\left(v^2(s,\cdot)\calu(s)\right) \ ds \notag \\
&= \int_{r}^{\infty} \left(2v(s)v^{\prime}(s,\cdot)\calu(s) + v^2(s,\cdot) \calu^{\prime}(s)\right) \ ds  \notag \\
&\leq C \int_{r}^{\infty}|v(s,\cdot)|\calu(s) \ ds \notag \\
&\leq C\int_{\{d(0,y)>r\}}|w(y,\cdot)-\calu(y)|\underbrace{\calu(y) \ \dvg(y)}_{\asymp \ 1}\
\end{align*}
where $\prime$ denotes the radial differentiation and we have used the bounds (see \eqref{decayvprime} and Section \ref{fd sec})
\begin{align*}
|\calu^{\prime}| \leq C\calu, \ |v^{\prime}|  + |v | \leq C, \ \dvg \asymp (\sinh s)^{n-1} ds \asymp \calu(s)^{-1} ds.
\end{align*}
Simple H\"{o}lder inequality gives
\begin{align*}
v(r,.)^2\calu(r) &\leq \left(\int_{\bn} |w - \calu|^{p+1} \ \dvg\right)^{\frac{1}{p+1}} \left(\int_{\{d(0,y)>r\}} \calu^{\frac{p+1}{p}}(y) \ \dvg(y)\right)^{\frac{p}{p+1}} \\
&\leq CE[w(\tau)]^{\frac{1}{p+1}}\left(\int_{r}^{\infty} e^{-(n-1)(\frac{p+1}{p} - 1)s} \ ds\right)^{\frac{p}{p+1}} \\
&\leq CE[w(\tau)]^{\frac{1}{p+1}}e^{-(n-1)(1 - \frac{p}{p+1})r},
\end{align*}
where in the second inequality we have used the fact that $|w - \calu|^{p+1} = (w - \calu)^{2} \, |w - \calu|^{p-1} \leq C(w - \calu)^{2} \calu^{p-1}.$

As a result, we get
\begin{align*}
|v(r)| \leq CE[w(\tau)]^{\frac{1}{2(p+1)}}e^{(n-1)\frac{p}{2(p+1)}r}\leq CE[w(\tau)]^{\frac{1}{2(p+1)}}e^{(n-1)\frac{p}{2(p+1)}r_0}.
\end{align*}

\medskip

\noindent
{\bf Step 5:} Conclusion. For the simplicity of notations, we denote 
\begin{align*}
\mathcal{E} = \sup_{\tau \in [\tau_0, \tau_1]} E[w(\tau)].
\end{align*}
Choose $\gamma > 0$ very small ($2p\gamma <1$ would suffice) 
and set $e^{(n-1)r_0} = \mathcal{E}^{-\gamma} > >1$ for $\tau_0$ large by Theorem \ref{entropy decay}. Let $x \in \bn$ be any arbitrary element.

If $r_0 < \frac{1}{2}d(x,0),$ then estimates of Step 3 gives
\begin{align*}
\int_{\tau_0}^{\tau_1}v(x,\tau) \ d\tau &\leq C(1 + |\tau_1 - \tau_0|)\left(e^{-(n-1)(p-2-\delta)r_0} + \mathcal{E}^{\frac{1}{2}}e^{(n-1)r_0}\right)\\
&\leq C(1 + |\tau_1 - \tau_0|)\left(\mathcal{E}^{\gamma(p-2-\delta)} + \mathcal{E}^{\frac{1}{2}-\gamma}\right) \\
&\leq C(1 + |\tau_1 - \tau_0|)\mathcal{E}^{\bar\kappa_1},
\end{align*}
where $\bar\kappa_1 := \min\{ \gamma(p-2-\delta), \frac{1}{2} - \gamma\}.$ On the other hand if $r_0 \geq \frac{1}{2}d(x,0),$ then $x \in \overline{B_{2r_0}(0)}$ and estimates of step 4 gives
\begin{align*}
v(x,\tau) &\leq C \mathcal{E}^{\frac{1}{2(p+1)}}e^{(n-1)\frac{p}{(p+1)}r_0} \\
&\leq C \mathcal{E}^{\frac{1}{2(p+1)} - \frac{p\gamma}{p+1}} = C \mathcal{E}^{\bar\kappa_2} \qquad\forall\, \tau \in[\tau_0, \tau_1],
\end{align*}
where $\bar\kappa_2 := \frac{1}{2(p+1)} - \frac{p\gamma}{p+1}.$ As a result the choice of $\gamma$ can be made by imposing $2p\gamma < 1.$
Choosing $\bar\kappa = \min\{\bar\kappa_1, \bar\kappa_2\}$ we get the desired result.

 \end{proof}

\noindent
{\bf Proof of Theorem \ref{linftyentropy}:} Using Lemma~\ref{intermediate} we get the following bounds: for every $\tau_1 \geq \tau_0 \geq \frac{1}{(1-m)T}\ln 2$ 

\begin{align*}
v(x, \tau_1) &\leq \frac{2m}{T(1-e^{-\frac{2m}{T}(\tau_1 - \tau_0)})} \int_{\tau_0}^{\tau_1} v(x,\tau) \ d\tau + \frac{2m^2}{T^2}(\tau_1 - \tau_0)^2\frac{1}{(1-e^{-\frac{2m}{T}(\tau_1 - \tau_0)})}
\notag \\
&\leq \frac{e^{\frac{2m}{T}(\tau_1-\tau_0)}}{\tau_1 - \tau_0}\int_{\tau_0}^{\tau_1} v(x,\tau) \ d\tau  + \frac{2m}{T}(\tau_1-\tau_0)e^{\frac{2m}{T}(\tau_1 - \tau_0)},
\end{align*}
where, to obtain the second inequality we have used the inequality  $1-e^{-\frac{2m}{T}(\tau_1 - \tau_0)}\geq \frac{2m}{T}(\tau_1 - \tau_0) e^{-\frac{2m}{T}(\tau_1 - \tau_0)}$.

On the other hand, for every $\tilde\tau_1 \geq \tilde\tau_0 \geq \frac{1}{(1-m)T}\ln 2$
\begin{align*}
v(x, \tilde\tau_0) &\geq -\frac{2m}{T(e^{\frac{2m}{T}(\tilde\tau_1 - \tilde\tau_0)}-1)} \Big |\int_{\tilde\tau_0}^{\tilde\tau_1} v(x,\tau) \ d\tau \Big|- \frac{2m^2}{T^2}(\tilde\tau_1 - \tilde\tau_0)^2\frac{e^{\frac{2m}{T}(\tilde\tau_1 - \tilde\tau_0)}}{(e^{\frac{2m}{T}(\tilde\tau_1 - \tilde\tau_0)}-1)} \notag \\
&\geq -\frac{1}{\tilde \tau_1 - \tilde \tau_0}\Big |\int_{\tilde\tau_0}^{\tilde\tau_1} v(x,\tau) \ d\tau \Big| - \frac{2m}{T}(\tilde \tau_1 - \tilde \tau_0)e^{\frac{2m}{T}(\tilde\tau_1 - \tilde\tau_0)},
\end{align*}
where, to obtain the second inequality we have used the inequality $e^{\frac{2m}{T}(\tilde\tau_1 - \tilde\tau_0)}\geq 1+\frac{2m}{T}(\tilde\tau_1 - \tilde\tau_0)$.

In either case, choosing $\tilde \tau_0 = \tau_1, \tilde \tau_1 - \tilde \tau_0 = \tau_1 - \tau_0$ and using Lemma~\ref{second-lemma} we can estimate 

\begin{align*}
 \left| \int_{s_0}^{s_1} v(x, \tau) \, {\rm d}t \right| \, \leq \, \left( K_1 \, + \, K_2 (\tau_1 - \tau_0) \right) \left( \sup_{\tau \in [\tau_0, 2\tau_1 - \tau_0]} E (w(\tau)) \right)^{\bar\kappa},
\end{align*}
where $s_0 = \tau_0$ or $\tilde \tau_0$ and $s_1 = \tau_1$ or $\tilde \tau_1.$
Combining all we get
\begin{align*}
 \|v(.,\tau_1) \|_{L^{\infty}(\bn)} \, \leq &\, K_1 \, \dfrac{e^{2m(\tau_1 - \tau_0)}}{\tau_1 - \tau_0} \left( K_1 \, + \, K_2 (\tau_1 - \tau_0) \right)\left(  \sup_{\tau \in [\tau_0, 2\tau_1-\tau_0]}E[w(\tau)] \right)^{\bar\kappa} \notag \\
 &+ \, \frac{2m}{T} (\tau_1 - \tau_0) e^{2m(\tau_1 - \tau_0)},
\end{align*}

for every $\tau_1 \geq \tau_0$ large enough. Now the result follows once we impose $\tau_1 - \tau_0 \leq 1.$
\hfill$\square$

\medskip

\subsection{Convergence in relative error with rate} With all this information from the last sub-section in hand, we can now state and prove the relative uniform decay rate.

\begin{theorem}\label{th:7.2}
Let $3 \leq n \leq 5,$ and $m \in (m_s, \frac{1}{2})$ and $v$ be the relative error defined by \eqref{error}. Then there exists a constant $\tilde\kappa = \tilde\kappa(n,m)$ and $C = C(n,m,u_0)$ such that 

\begin{align*}
\|v(.,\tau)\|_{L^{\infty}(\bn)} \leq Ce^{-\tilde\kappa\tau}
\end{align*}
for all $\tau \geq \tau_0$ with $\tau_0$ large enough.
\end{theorem}

\begin{proof}
By Theorem \ref{entropy decay} we know that the linear entropy decays exponentially: there exists $\kappa >0$ such that 
\begin{align*}
E[w(\tau)] \leq C e^{-\frac{2\kappa}{p+1}\tau}, \ \ \tau \geq \tau_0
\end{align*}
for $\tau_0$ large enough. By enlarging $\tau_0$ we can assume $E[w(\tau)] \leq 1$ for all 
$\tau \geq \tau_0.$ Set 
\begin{align*}
\tau_1 = \tau_0 + \left(\sup_{\tau \in [\tau_0, \tau_0 + 2]} E[w(\tau)]\right)^{\frac{\bar\kappa}{2}},
\end{align*}
where $\bar\kappa$ is as in Theorem \ref{linftyentropy}. Then 
\begin{align*}
\tau_1 - \tau_0 \leq 1, \ \ 2 \tau_1 - \tau_0 \leq \tau_0 + 2, \ \ e^{2m(\tau_1 - \tau_0)} \leq e^{2m}.
\end{align*}
Applying Theorem \ref{linftyentropy} with this choice of parameters we get
\begin{align*}
\|v(.,\tau_1)\|_{L^{\infty}(\bn)} &\leq (K_1+2m)e^{2m} \left(\sup_{\tau \in [\tau_0 , \tau_0+2]} E[w(\tau)]\right)^{\frac{\bar\kappa}{2}}, \notag \\
&\leq C((K_1+2m)e^{2m}) e^{-\frac{\kappa\bar\kappa}{p+1}\tau_0} \notag \\
&\leq C((K_1+2m)e^{2m}) e^{-\frac{\kappa\bar\kappa}{p+1}\tau_1} e^{\frac{\kappa\bar\kappa}{p+1}(\tau_1-\tau_0)}\notag \\
&\leq \tilde 
Ce^{-\frac{\kappa\bar\kappa}{p+1}\tau_1}
\end{align*}
with $\tilde C = C((K_1+2m)e^{2m}) e^{\frac{\kappa\bar\kappa}{p+1}}.$ This completes the proof with $\tilde\kappa = \frac{\kappa\bar\kappa}{p+1}.$
\end{proof}

\medskip

\noindent
{\bf Proof of Theorem \ref{asymptotic}.} The proof directly follows from Theorem~\ref{th:7.2} by substituting the definition of $v$ and using the relation between $\tau$ and $t$. Here $\mu=\frac{\tilde\kappa}{1-m}$ and $\tilde\kappa$ is as found in Theorem~\ref{th:7.2}.
\hfill{$\square$}

\medskip

\section{Stability results for Hardy-Sobolev-Maz'ya inequalities}\label{hsm sec}

In this section, we establish the stability results of the Hardy-Sobolev-Mazy'a inequality in the spirit of Theorem~\ref{maintheorem-1}. First, we recall Hardy-Sobolev-Maz'ya (HSM) inequality \cite[2.1.6 Corollary~3]{VGM}: 
 
 \medskip
 
 \noindent 
 {\bf HSM-inequality.} Let $N \geq 3,$ we fix parameters $k, h \in \mathbb{N}$ such that $2 \leq k < N, \ N = k + h,$ and let $1 < p \leq \frac{N+2}{N-2}$ be given and set $t = N - \frac{(N-2)(p+1)}{2} \geq 0.$ We denote a generic point of $ \mathbb{R}^N$ by $ (y,z)$ where $y \in \mathbb{R}^k$ and $z \in \mathbb{R}^h.$ The HSM-inequality asserts that for every $\mu \leq \frac{(k-2)^2}{4},$ there exists a best constant $\mathcal{S}(N,k,p,\mu)$ such that
 


\begin{align}\label{inq-HSM2}
\mathcal{S}(N,k,p,\mu)
\left( \int_{\mathbb{R}^k \times \mathbb{R}^h} \frac{|v|^{p+1}}{|y|^t}  \,  {\rm d}y \, {\rm d}z\right)^{\frac{2}{p+1}} \leq
\int_{\mathbb{R}^k \times \mathbb{R}^h} \left( |\nabla v|^2 \,- \,  \mu \, \frac{v^2}{|y|^2} \right) \,  {\rm d}y \, {\rm d}z,
\end{align}
holds for every $v \in C_c^{\infty}(\mathbb{R}^k \times \mathbb{R}^h)$. By density \eqref{inq-HSM2} can be extended to the closure with respect to the norm defined on the right-hand side of \eqref{inq-HSM2}.


Under the assumption $0 \leq \mu \leq \frac{(k-2)^2}{4},$ the value $\mathcal{S}(n,p,\mu)$ does not change if we restrict to class of functions which are symmetric about $y$-variable, i.e., $v(y,z) = v(|y|,z)$ by abuse of notations \cite{GM, MS, FMS} (see also \cite{GhR} for results in bounded domain with $\mu = 0$). From now on we will restrict ourselves to such classes.
In particular, \eqref{inq-HSM2} holds for all 
 $v $ in the homogeneous Sobolev space $D^{1, 2}_{cyl}(\mathbb{R}^k \times \mathbb{R}^h)$ defined below and we will study the stability problem on this space.

\medskip

\begin{definition}
{\rm
We say a smooth function $v$ on $(\mathbb{R}^k \setminus \{ 0 \}) \times \mathbb{R}^h$ is  symmetric in $y$-variable if for any choice of $z \in \mathbb{R}^{h},$ $v(., z)$ is symmetric decreasing in $\mathbb{R}^k.$ We define $D^{1, 2}_{cyl}(\mathbb{R}^k \times \mathbb{R}^h)$ by the closure of 
\begin{align*}
\{v\in C_c^{\infty}((\mathbb{R}^k \setminus \{ 0 \}) \times \mathbb{R}^h) : v := v(|y|, z) \}
\end{align*}
with respect to the norm 
\begin{align*}
\| v \|_{\mu, \ D^{1, 2}_{cyl}(\mathbb{R}^k \times \mathbb{R}^h)} := \left( \int_{\mathbb{R}^k \times \mathbb{R}^h}\left[ |\nabla v|^2 \,-\, \mu \, \frac{v^2}{|y|^2} \right] \, {\rm d}y\, {\rm d}z   \right)^{\frac{1}{2}},
\end{align*}
where $0 \leq \mu \leq \frac{(k-2)^2}{4}.$ } If a function is symmetric in both $y$ and $z$ variable then we call it 
cylindrically symmetric.
\end{definition}
\medskip

Note that if $k=2$ then $\mu = 0$ and for other $k$ whenever $\mu < \frac{(k-2)^2}{4},\ \| \cdot \|_{\mu, D^{1, 2}_{cyl}(\mathbb{R}^k \times \mathbb{R}^h)}$ is equivalent Sobolev norm on the class of functions symmetric with respect to $y$-variable.
It is straightforward to derive that up to normalization, the Euler-Lagrange equation associated with the Hardy-Sobolev-Maz'ya inequality \eqref{inq-HSM2} is given by

\begin{equation}\label{eq-HSM}
-\Delta v(y,z) \; - \; \mu \; \frac{v(y,z)}{|y|^2} \; = \; \frac{|v(y,z)|^{p-1} v(y,z)}{|y|^t}, \quad x = (y, z) \in \mathbb{R}^k \times \mathbb{R}^h,
\end{equation}
where $t, p,\mu$ satisfy the same assumptions as in \eqref{inq-HSM2}.

\medskip

We define the HSM-deficit: for $v \in D^{1, 2}_{cyl}(\mathbb{R}^k \times \mathbb{R}^h)$
\begin{align}
\delta_{\mbox{\tiny{HSM}}}^{\mu} (v) = \left[\int_{\mathbb{R}^k \times \mathbb{R}^h} \left( |\nabla v|^2 \,- \,  \mu \, \frac{v^2}{|y|^2} \right) \,  {\rm d}y \, {\rm d}z - \mathcal{S}(N,k,p,\mu)\left( \int_{\mathbb{R}^k \times \mathbb{R}^h} \frac{|v|^{p+1}}{|y|^t}  \,  {\rm d}y \, {\rm d}z\right)^{\frac{2}{p+1}} \right]^{\frac{1}{2}}
\end{align}
and the corresponding Euler-Lagrange functional 

\begin{align*}
J_{\mbox{\tiny{HSM}}}^{\mu}(v) =  \frac{1}{2}\int_{\mathbb{R}^k \times \mathbb{R}^h}\left( |\nabla v|^2 \,-\, \mu \, \frac{v^2}{|y|^2} \right) \, {\rm d}y\, {\rm d}z   - \frac{1}{p+1} \int_{\mathbb{R}^k \times \mathbb{R}^h} \frac{|v|^{p+1}}{|y|^t}\, {\rm d}y\, {\rm d}z.
\end{align*}

To study the stability of \eqref{inq-HSM2} or \eqref{eq-HSM}, we must first ensure the uniqueness of positive solutions, up to the respective invariants. 

\medskip

The symmetry properties of positive solutions to \eqref{eq-HSM} is a delicate issue and even endure symmetry breaking in the $y$ variable if $\mu <0$ sufficiently large \cite[Theorem 0.4]{GM}. Indeed, the ground state solutions suffer symmetry breaking in $y$-variable for $\mu \leq \frac{(k-2)^2}{4} - \frac{k-1}{p-1} < 0$ for the sub-critical $p$ and $\mu <0$ for the critical $p$. However, the situation is much better for $\mu \geq 0.$ Using the rearrangement techniques such as the product of symmetrization, it has been proved in \cite{BT, SSW} that for $\mu =0,$ the best constant in \eqref{inq-HSM2} does not change if one restricts to the class of cylindrically symmetric functions (i.e., symmetric in both $y$ and $z$ variable). Regarding the Euler-Lagrange equation \eqref{eq-HSM}, it is first shown in \cite{ FMS, MS} for the special case $\mu = 0, t = 1$ (i.e., $p = \frac{2(N-1)}{N-2}$) that positive entire solutions are cylindrical symmetric. To the best of our knowledge, the best-known result to date on the symmetry properties is obtained by Gazzini and Musina in \cite{GM} which asserts that 
for $0 \leq \mu \leq \frac{(k-2)^2}{4}$ and $1 < p \leq \frac{N+2}{N-2}$ all positive, classical, finite energy solutions are cylindrically symmetric.

\medskip

A new pathway of studying extremals of inequalities
 \eqref{inq-HSM2} (or more generally, studying the Euler-Lagrange equation \eqref{eq-HSM}) has been proposed by Sandeep-Mancini in  \cite{MS} by lifting the problem (restricted to the class $D^{1, 2}_{cyl}(\mathbb{R}^k \times \mathbb{R}^h)$) on the upper half-space model of the hyperbolic space.

\medskip

\subsection{The upper half space model of the hyperbolic space}
 $\hn = \mathbb{R}_+\times \mathbb{R}^{n-1} := \{x=(r,z) \ | \ r >0, \ z \in \R^{n-1}\}$ endowed with the Riemannian metric $\frac{dx^2}{r^2}$  
 represents the upper half space model of the hyperbolic $n$-space. The volume element, the gradient vector field and the Laplace-Beltrami operator are given by
 \begin{align*}
 {\rm d}v_{\hn} = \frac{{\rm d}x}{r^n}, \ \ \ \nabla_{\hn} = r^2\nabla, \ \ \ \Delta_{\hn} = r^2\Delta - (n-2)r\partial_{r}
 \end{align*}
where $\partial_{r}$ is the derivative with respect to the $r$ variable and $dx, \ \nabla, \ \Delta$ are the Euclidean volume, gradient and Laplace operator respectively.

\subsection{Isometric lifting on the hyperbolic space}  
 Using a suitable transformation described below
 the authors of \cite{MS}  lifted the problem to the real hyperbolic space of the upper-half space model $\hn:$ 
 
 \medskip
 
 Given $n \geq 3$ and $0 \leq \mu < \frac{(k-2)^2}{4}$ we fix  parameters 
 \begin{align}  \label{several relations}
 \begin{cases}
 h = n-1, \ \ \ N = k+h, \ k \geq 3, \ \ \ \lambda = \mu + \frac{(n-1)^2 - (k -2)^2}{4}, \\ 
 \\
1 < p \leq \frac{N+2}{N-2}, \ \mbox{and} \ t = N - \frac{N-2}{2}(p+1).
 \end{cases}
 \end{align}

 Let $w (y, z) = w(|y|, z)$ be a  symmetric in $y$-variable solution of \eqref{eq-HSM}, and define
 $u(r, z) := r^{\frac{n-2}{2}} w(r, z).$ Then $u$ solves  the following problem: 
 \begin{equation}\label{trns-hyp}
 -\Delta_{\mathbb{H}^n} u - \lambda \, u = |u|^{p-1}u, \quad u \in H^1(\mathbb{H}^n).
 \end{equation}
 According to \cite{MS}, solutions of \eqref{trns-hyp} are unique up to hyperbolic isometries. Let $I(\hn)$ denote the isometry group of $\hn,$ then there exists a unique solution, denoted by $\calu,$ symmetric in $z$-variable such that $\{\calu \circ \tau \ | \ \tau \in I(\hn)\}$ are all the solutions to \eqref{trns-hyp}. In addition, the corresponding Sobolev spaces $H^1(\mathbb{H}^n)$ has an isometry with 
$D^{1, 2}_{cyl}(\mathbb{R}^k \times \mathbb{R}^h)$ and vice-versa:

\begin{itemize}
\item[$\bullet$]  Define the map $\mathcal{T}_k : C_c^{\infty}(\hn) \rightarrow D^{1, 2}_{cyl}(\mathbb{R}^k \times \mathbb{R}^h)$ as follows: let $\varphi \in C_c^{\infty}(\hn),$ we define the symmetric in $y$-variable function $\mathcal{T}_k(\varphi)(y,z) = r^{-\frac{N-2}{2}}\varphi(r,z),$ where $(y,z) \in \mathbb{R}^k \times \mathbb{R}^h, \ r = |y|.$ Then a straight forward computation gives for $\varphi, \psi \in C_c^{\infty}(\hn)$

\medskip

\begin{align}\label{several computations}
&\frac{1}{\omega_k} \int_{\mathbb{R}^k \times \mathbb{R}^h} \frac{\mathcal{T}_k(\varphi) \mathcal{T}_k(\psi)}{|y|^2} \, {\rm d}y \, {\rm d}z \, =\, \int_{\mathbb{H}^{n}} \varphi \psi \, {\rm d}v_{\hn}, \notag\\
\notag\\
&\frac{1}{\omega_k} \int_{\mathbb{R}^k \times \mathbb{R}^h} \frac{|\mathcal{T}_k(\varphi)|^{p-1}\mathcal{T}_k(\varphi)\mathcal{T}_k(\psi)}{|y|^t} \, {\rm d}y \, {\rm d}z \, =\, \int_{\mathbb{H}^{n}} |\varphi|^{p-1}\varphi \psi \, {\rm d}v_{\hn}, \notag\\
\notag\\
&\frac{1}{\omega_k} \int_{\mathbb{R}^k \times \mathbb{R}^h} \nabla \mathcal{T}_k(\varphi) \cdot \nabla \mathcal{T}_k(\psi)\, {\rm d}y \, {\rm d}z \, = \, \int_{\mathbb{H}^{n}} \left( \langle \nabla_{\hn} \varphi, \nabla_{\hn}\psi \rangle_{\hn} - \frac{h^2 - (k-2)^2}{4} \, \varphi\psi \right) \, {\rm d}v_{\hn},
\end{align}
\end{itemize}

\medskip

where $\omega_k = |\mathbb{S}^{k-1}|$ the surface measure of the $(k-1)$-dimensional sphere. As a result by density 
\begin{align*}
\frac{1}{\omega_k} J_{\mbox{\tiny{HSM}}}^{\mu}(\mathcal{T}_k(\varphi)) = I_{\lambda}(\varphi), \ \ \frac{1}{\omega_k}\delta_{\mu}(\mathcal{T}_k(\varphi)) = \delta_{\lambda}(\varphi).
\end{align*}
holds for every $\varphi \in H^1(\hn),$ where $I_{\lambda}$ is defined analogous to the Euler-Lagrange functional \eqref{ilambda} with $\bn$ replaced by  $\hn:$
\begin{align*}
I_{\lambda}(u) = \frac{1}{2}\int_{\mathbb{H}^{n}} \left( |\nabla_{\hn} u|^2 - \lambda \, u^2 \right) \, {\rm d}v_{\hn} - \frac{1}{p+1}\int_{\mathbb{H}^{n}} |u|^{p+1} \, {\rm d}v_{\hn}
\end{align*}
and $\delta_{\lambda}$ is the deficit functional in Poincar\'{e}-Sobolev inequality with an emphasis on the dependence on $\lambda$.
Since $N > n$ we see that $1<p \leq \frac{N+2}{N-2} < \frac{n+2}{n-2}$ 
 and hence existence and uniqueness of \eqref{trns-hyp} applies and translates into the existence and uniqueness of cylindrically symmetric entire positive solutions of \eqref{eq-HSM}. In particular, under the assumption \eqref{several relations}
 all the solutions to \eqref{eq-HSM} are cylindrically symmetric about some point and hence solutions are unique up to scaling and translation in $\mathbb{R}^h.$ In other words, let $\mathcal{V}$ be the unique solution which is symmetric and decreasing about both $y$ and $z$ variable  then $\mathcal{V}_{R, z_0}(y, z) := R^{\frac{n-2}{2}}\mathcal{V}(Ry, Rz + z_0)$ forms the solution space for \eqref{eq-HSM} and hence 
 \begin{align*}
 \tilde{\mathcal{Z}_0} :=  \{ c\, \mathcal{V}_{R, z} : c \in \mathbb{R}, R \in \mathbb{R}^{+}, z \in \mathbb{R}^h \}
 \end{align*}
froms the $(h+2)$-dimensional manifold consisting of all the extremizers of \eqref{inq-HSM2}. For  $v \in D^{1, 2}_{cyl}(\mathbb{R}^k \times \mathbb{R}^h)$ we define the distance 
\begin{align*}
dist(v, \tilde{\mathcal{Z}_0}) = \| v - c\, \mathcal{V}_{R, z}\|_{\mu, \ D^{1, 2}_{cyl}(\mathbb{R}^k \times \mathbb{R}^h)}
\end{align*}

 \begin{remark}
 {\rm
 It is important to note that the equation~\ref{eq-HSM} admit ground state solutions for some $\mu < <0$ and $1< p < \frac{N+2}{N-2},$ which are not cylindrically
  symmetric \cite{GM}.
  }
 \end{remark}
 \medskip

Now we state our main stability result for Hardy-Sobolev-Maz'ya inequality.

\begin{theorem}\label{application-ineqHSM}
Let $N \geq 5, \ 3 \leq k < N, \ h \geq 2$ be such that $N = k + h.$ Let $1 < p \leq \frac{N+2}{N-2},$  and set $t = N - \frac{(N-2)(p+1)}{2}$ and let $0 \leq \mu < \frac{(k-2)^2}{4}.$
 There exists a constant $C=C(N, p, \mu,k)$ such that for 
all $v \in D^{1, 2}_{cyl}(\mathbb{R}^k \times \mathbb{R}^h)$ there holds 
\begin{align*}
 \hbox{dist} \, (v, \tilde{\mathcal{Z}_0})) \leq  C \delta_{\mbox{\tiny{HSM}}}^{\mu}(v).
\end{align*} 

\end{theorem}

\medskip 

Next, we establish a theorem in the spirit of Theorem~\ref{maintheorem-2}. Namely, if a non-negative function $v \in D^{1, 2}_{cyl}(\mathbb{R}^k \times \mathbb{R}^h)$ almost solves 
\eqref{eq-HSM}, then $v$ is close to $\mathcal{V}_{R, z}$ for some $R \in \mathbb{R}^+$ and $z \in \mathbb{R}^h$ in a quantitative way.

 We know the critical points of the energy functional $J_{\mbox{\tiny{HLS}}}^{\mu}$ are the solutions to \eqref{eq-HSM}. For topological vector space $X,$ let $X^{\prime}$ denotes its dual.  Now we state our main result. 

\begin{theorem}\label{application-eqHSM}
Let $N \geq 5, \ 3 \leq k < N, \ h \geq 2$ be such that $N = k + h.$ Let $1 < p \leq \frac{N+2}{N-2},$  and set $t = N - \frac{(N-2)(p+1)}{2}$ and let $0 \leq \mu < \frac{(k-2)^2}{4}.$
 Let $\{ v_{\alpha} \}$ be a Palais-Smale 
sequence of non-negative functions for $J_{\mbox{\tiny{HLS}}}^{\mu}$ such that  $\lim_{\alpha \rightarrow \infty} J_{\mbox{\tiny{HLS}}}^{\mu}(v_{\alpha}) =  \frac{p-1}{2(p+1)} \mathcal{S}(N,k, p, \mu),$
then there exists a constant $C=C(N,k, p, \mu) > 0$ such that 
\begin{align*}
\hbox{dist}(v_{\alpha}, \tilde{\mathcal{Z}} )  \leq C \| J_{\mu}^{\prime} (v_{\alpha}) \|_{(D^{1, 2}_{cyl}(\mathbb{R}^k \times \mathbb{R}^h))^{\prime}},
\end{align*}
where  $\tilde{\mathcal{Z}} =  \{ \mathcal{V}_{R, z} :  R \in \mathbb{R}^{+}, z \in \mathbb{R}^h \}
$ denotes the manifolds of non-negative cylindrically symmetric solutions of \eqref{eq-HSM}.

\end{theorem}

\medskip 

\noindent
{\bf Proof of Theorem~\ref{application-ineqHSM}.} Let $n = h+1$ then we have $1 < p \leq \frac{N+2}{N-2} < \frac{n+2}{n-2}.$ 
Let $\mathcal{T}$ denotes the map which lifts an element $v$ of $D^{1, 2}_{cyl}(\mathbb{R}^k \times \mathbb{R}^h)$ to $H^1(\hn)$ according to the formula 
\begin{align*}
\mathcal{T}(v)(r, z) = r^{\frac{N-2}{2}}v(|y|, z), \ r = |y|, \ (y,z) \in \mathbb{R}^k \times \mathbb{R}^h).
\end{align*}
Let $\mathcal{V}$ be the unique solution to \eqref{eq-HSM} symmetric about both $y$ and $z$ variable and let $\mathcal{V}_{R,z_0} (y,z)= R^{\frac{N-2}{2}}\mathcal{V}(Ry, Rz + z_0)$ where $R>0$ and $z_0 \in \mathbb{R}^h.$ Then 
$\mathcal{T}(\mathcal{V}_{R,z_0})$ solves \eqref{trns-hyp} and hence by uniqueness there exists as isometry $\tau \in I(\hn)$ of the hyperbolic space $\hn$ such that $\mathcal{T}(\mathcal{V}_{R,z_0}) = \calu \circ \tau$ where $\calu$
 is the unique radial solution to \eqref{trns-hyp}. Conversely, for every $\tau$ in the isometry group $I(\hn)$ of $\hn$, $T_k(\calu \circ \tau)$ gives rise to a finite energy solution to \eqref{eq-HSM} where $\mathcal{T}_k$ is as in \eqref{several computations}, and hence by symmetry properties of solutions to \eqref{eq-HSM}, there is a one to one correspondence between $\mathcal{V}_{R,z_0}$ and $\calu \circ \tau.$ 
 Now for any $v \in D^{1, 2}_{cyl}(\mathbb{R}^k \times \mathbb{R}^h)$ by \eqref{several computations} we have 
 
 \begin{align*}
 \|v - \mathcal{V}_{R,z_0}\|_{\mu, D^{1, 2}_{cyl}(\mathbb{R}^k \times \mathbb{R}^h)} = \|\mathcal{T}(v) - \calu \circ \tau\|_{\lambda}
 \end{align*}
where $\lambda$ and $\mu$ are related by \eqref{several relations} and our assumption on $\mu$ implies $\lambda < \frac{(n-1)^2}{4}$. As a result, our stability result applies: Theorem \ref{maintheorem-1} imply
\begin{align*}
\inf_{\tau \in I(\hn)} \|\mathcal{T}(v) - \calu \circ \tau\|_{\lambda} \leq C \delta(\mathcal{T}(v))
\end{align*}
and according to the formulas \eqref{several computations}, $\delta(\mathcal{T}(v)) = \delta_{\mbox{\tiny{HLS}}}^{\mu} (v).$
This completes the proof of the theorem.





\medskip 

{\bf Proof of Theorem~\ref{application-eqHSM}} The proof is essentially similar to the Theorem \ref{application-ineqHSM}.

For $v , w\in  D^{1, 2}_{cyl}(\mathbb{R}^k\times \mathbb{R}^k),$ using \eqref{several computations} we have

\begin{align*}
\int_{\mathbb{R}^k \times \mathbb{R}^h} \left(-\Delta w - \frac{\mu}{|y|^2}v\right)w \, {\rm d}y \, {\rm d}z
&= \int_{\mathbb{R}^k \times \mathbb{R}^h} \left(\nabla w \cdot \nabla w - \frac{\mu}{|y|^2}vw\right) \, {\rm d}y \, {\rm d}z \\
&= \omega_k \int_{\hn} \left( \langle \nabla_{\hn}\mathcal{T}(w), \nabla_{\hn} \mathcal{T}(w)\rangle_{\hn} - \lambda \mathcal{T}(v)\mathcal{T}(w)\right) \, {\rm d}v_{\hn}\\
&=\omega_k \int_{\hn} \left(-\Delta_{\hn} \mathcal{T}(v)-\lambda \mathcal{T}(v) \right)\mathcal{T}(w) \, {\rm d}v_{\hn}.
\end{align*}
Since $\mathcal{T}$ is an isometry between the two Sobolev spaces the result follows by Theorem \ref{maintheorem-2}.

\hfill$\square$

\bigskip

  \par\bigskip\noindent
\textbf{Acknowledgments.}
The research of M.~Bhakta is partially supported by the {\em SERB WEA grant (WEA/2020/000005)} and DST Swarnajaynti fellowship (SB/SJF/2021-22/09). D.~Ganguly is partially supported by the INSPIRE faculty fellowship (IFA17-MA98). D.~Karmakar acknowledges the support of the Department of Atomic Energy, Government of India, under project no. 12-R\&D-TFR-5.01-0520. S.~Mazumdar is partially supported by IIT Bombay SEED Grant RD/0519-IRCCSH0-025.


\end{document}